\documentclass[11pt]{article}
\usepackage{amsthm}
\usepackage{amssymb}
\usepackage{epsfig}
\usepackage{epstopdf}
\usepackage[dvipsnames,usenames]{color}
\usepackage{paralist}
\usepackage{amsmath}
\usepackage{amsfonts}
\usepackage{float}
\usepackage{xcolor}
\usepackage{lpic}
\usepackage[all]{xy}
\usepackage[hidelinks]{hyperref}
\usepackage{mathabx} 
\usepackage{graphicx}
\usepackage{caption}
\usepackage{subcaption}
\usepackage{tikz}

\usetikzlibrary{cd}

\newcounter{contador}
\newcounter{teoA}

\newtheorem{teoa}[teoA]{Theorem}

\newtheorem{propo}[contador]{Proposition}
\newtheorem{teo}[contador]{Theorem}
\newtheorem{lem}[contador]{Lemma}
\newtheorem{defi}[contador]{Definition}

\theoremstyle{remark}

\newcounter{ex}

\newcommand{\R}{{\mathbb R}}

\newcommand{\N}{{\mathbb N}}
\newcommand{\Q}{{\mathbb Q}}

\title{Entropy and $\omega$-limit sets on invariant graphs for a class of planar piecewise linear maps}

\author{Anna Cima$^{(1)}$, Armengol Gasull$^{(1)}$, V\'{\i}ctor Ma\~{n}osa$^{(2)}$ and Francesc Ma\~{n}osas$^{(1)}$
    \\*[.1truecm]
    {\small \textsl{$^{(1)}$ Departament de Matem\`{a}tiques, Facultat
            de Ci\`{e}ncies,}}
    \\*[-.25truecm] {\small \textsl{Universitat Aut\`{o}noma de Barcelona,}}
    \\*[-.25truecm] {\small \textsl{08193 Bellaterra, Barcelona,
    Spain}}
    \\*[-.25truecm] {\small \textsl{anna.cima@uab.cat,
            armengol.gasull@uab.cat, francesc.manosas@uab.cat}}\\
    \\*[-.25truecm] {\small \textsl{$^{(2)}$ Departament de Matem\`{a}tiques,}}
     \\*[-.25truecm] {\small \textsl{Institut de Matem\`{a}tiques de la UPC-BarcelonaTech (IMTech),}}           
    \\*[-.25truecm] {\small \textsl{Universitat Polit\`{e}cnica de Catalunya}}
    \\*[-.25truecm] {\small \textsl{Colom 11, 08222 Terrassa, Spain}}
    \\*[-.25truecm] {\small \textsl{victor.manosa@upc.edu}}}

\date{\today}

\begin{document}

\maketitle

\begin{abstract}  We consider 
the family of piecewise linear maps
$F(x,y)=\left(|x| - y + a, x - |y| + b\right),$
where $(a,b)\in \R^2$.  In previous work, we identified that certain maps of this class possess one-dimensional invariant sets, planar graphs, that capture the global dynamics of the system. Within these graphs, chaotic dynamics emerge for certain parameter values, leading to an intermediate dynamical regime between regular behavior and full-plane chaos. In the present study, we revisit this family and analyze in detail the topological entropy as a function of a bifurcation parameter, finding that transitions from positive to zero entropy are continuous for certain parameter values and discontinuous for others. We also provide a methodology for determining arbitrarily sharp rational bounds for the bifurcation values at which this transition occurs. Finally, motivated by the limitations of numerical simulations in detecting the complex dynamics within these graphs, we prove that for some parameter values, there exists a full-measure set in these graphs where orbits converge to at most three omega-limit sets, which, when the parameter values are rational, correspond to periodic orbits.
\end{abstract}



\newpage

\section{Introduction and main results}

Planar continuous piecewise linear maps have long served as paradigmatic models in nonlinear dynamics. Since the foundational works of Lozi and Devaney \cite{D84,L78}, systematic studies of their dynamics and of bifurcations in such systems, and more generally in piecewise smooth maps, were initiated, for instance,  through the theory of border-collision bifurcations developed by Nusse and Yorke for one-dimensional maps \cite{NY92,NY95}. Since then, chaotic dynamics and bifurcation structures in two-dimensional piecewise linear maps have been widely investigated; the literature over the last decades is vast, to cite just a few examples see \cite{BG99,GSA,GT20,GMS,TG20}.

In our previous work~\cite{CGMM}, we observed that there exist non-trivial continuous piecewise linear maps that possess one-dimensional invariant sets--specifically, planar graphs, that is, a finite set of points (vertices) connected by a finite set segments (edges), each joining some pairs of them, that capture the global dynamics in the plane; we proved the existence of these graphs and also that for certain parameter values, chaotic dynamics arise on these graphs, producing an intermediate dynamical regime between regular behavior and fully developed chaos in the plane. Of course, trivial examples can be obtained by embedding in the plane a one-dimensional map whose global dynamics is contained in an interval, but these examples do not truly exhibit two-dimensional dynamics.

One-dimensional attractors consisting of line segments in continuous piecewise-linear maps have been observed numerically in the study of border-collision bifurcations since at least the work of Parui and Banerjee \cite{PB}. When the first version of this manuscript was prepared, one of the anonymous reviewers drew our attention to the paper by Kowalczyk \cite{Kow}. We consider this a relevant contribution. This reference studies border-collision bifurcations in non-invertible piecewise-linear planar maps, which serve as normal forms for grazing-sliding bifurcations in Filippov-type systems. The author derives, analytically,  conditions under which the considered maps exhibit the sudden onset of robust chaotic attractors whose topological structure consists of two or three distinct continuous line segments forming \emph{a tree}; see Propositions 1 and 3 of \cite{Kow}.

The family of piecewise linear maps introduced in  our previous work, and which we also consider in the present paper, is given by
\begin{equation}\label{e:F}
F_{a,b}(x,y)=\left(|x| - y + a, x - |y| + b\right),
\end{equation}
where $(a,b)\in \R^2$.  
 In what follows, whenever there is no risk of confusion, we will use $F$ instead of $F_{a,b}$ to denote the map. In the last years, some works have appeared that analyze different particular cases of the family $F$, motivated by the fact that this family intersects the class of maps introduced by Grove and Ladas in \cite{GL} with the aim of generalizing the  Lozi map, see for example \cite{ABK21,BS21, BK,TLL13, TLS17}. These works characterize particular cases in which  every orbit  converges to a fixed point, to a periodic orbit, or it is pre-periodic (that is, their points reach a periodic orbit in a finite number of iterations). Our initial motivation for studying this family was precisely to analyze the structure of these periodic and preperiodic orbits (their set of periods, the possible coexistence, the shape of their basins of attraction when coexistence occurs...)
However, we found that the global dynamics of the family  $F$ is remarkably more complex. Although the family of maps $F$ does not belong to the class of Lozi-like maps in the sense of Misiurewicz and \v{S}timac \cite{MS1,MS2}, we addressed similar questions to those considered by these authors, such as how topological entropy and periodic points depend on the parameters, or the nature of the attractors, to mention only a few.

In this study, we revisit this family and analyze in detail the topological entropy as a function of a bifurcation parameter, finding that transitions from positive to zero entropy occur continuously-whereas, we previously found that transitions from zero to positive entropy are discontinuous. We also provide a methodology for determining arbitrarily sharp rational bounds for the bifurcation values at which this transition occurs. Finally, motivated by the limitations of numerical simulations in detecting the complex dynamics within these graphs, we prove that for some parameter values, there exists a full-measure set in these graphs where orbits converge to at most three omega-limit sets, which, when the parameter values are rational, correspond to periodic orbits. 

In the following subsections, we first revisit the main results from~\cite{CGMM}, included here for completeness, and then present the  contributions of this paper.

\subsection{Dynamics of the family \eqref{e:F}}

In the paper \cite{CGMM}, we presented a study  of the family \eqref{e:F}. 
For completeness, we reproduce the detailed statements of the main results from the referenced work (Theorems~\ref{t:teoA}--\ref{t:teoD}, below).  We include this section to provide the reader with the context in which the results presented in this work are framed. As shown in Section \ref{s:s2}, the family $F_{a,b}$ is essentially one-parameter, since the re-scaling \eqref{conj} allows the parameter $a$ to be normalized. Thus, the cases $a>0$ and $a<0$ reduce to $a=1$ and $a=-1$, respectively, whereas the case $a=0$ must be treated separately. Note, however, that the cases $a=1$ and $a=-1$ are not conjugate.

The first result characterizes completely the  dynamics of $F$   when $a\ge0.$

\begin{teoa}\label{t:teoA}
	If $a\ge 0,$  for each $\mathbf{x}\in\R^2$   there exists $n\ge 0,$ that may depend on $\mathbf{x},$ such that $F^n(\mathbf{x})\in \operatorname{Per}(F),$ \emph{the set  of all periodic points of $F$}. Moreover, the set $\operatorname{Per}(F)$ is formed by a fixed point and, depending on $a$ and $b,$ either two or none 3 periodic orbits.
	\end{teoa}

The most interesting scenario concerning dynamics arises when $ a < 0 $. To analyze it, we first demonstrate that, in this case, the dynamics is confined in an one-dimensional graph.

\begin{teoa}\label{t:teoB}
	 If $a<0,$ for each $b\in\R$  there is a compact graph $\Gamma=\Gamma_{a,b}$,  which is invariant under the map $F,$ such that for every $\mathbf{x}\in \R^2$ there exists a non-negative integer $n,$ that may depend on $\mathbf{x},$ such that $F^n(\mathbf{x})\in \Gamma.$  
\end{teoa}

We remark that the graphs $\Gamma$ referred to in Theorem \ref{t:teoB} are unique and were completely characterized in \cite{CGMM}. In fact, in the appendix of that paper, we listed all the graphs arising in the family, obtaining a total of 37 distinct graphs (that include the boundary cases which explain the transition from one graph to another). In Figures \ref{f:esp3},  \ref{particio-1-3/4}, \ref{f:A} and \ref{ff:1}, we display the graphs corresponding to the parameter values considered in the statements of several results in the paper. In Figure \ref{fig:graphs}, we present a selection of graphs illustrating the topological diversity of these objects.

\begin{figure}[H]
\centering

\includegraphics[width=0.31\textwidth]{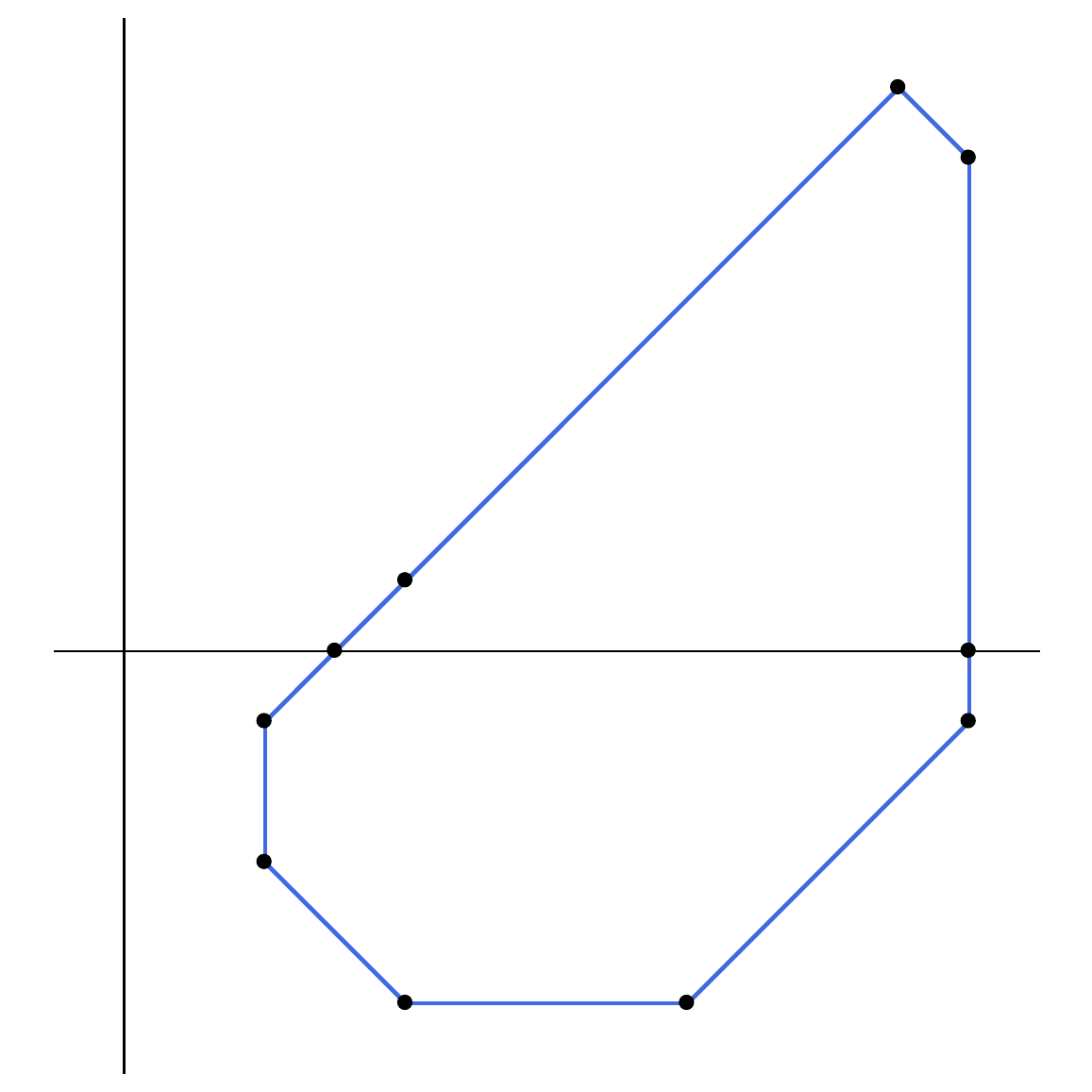}\hfill
\includegraphics[width=0.31\textwidth]{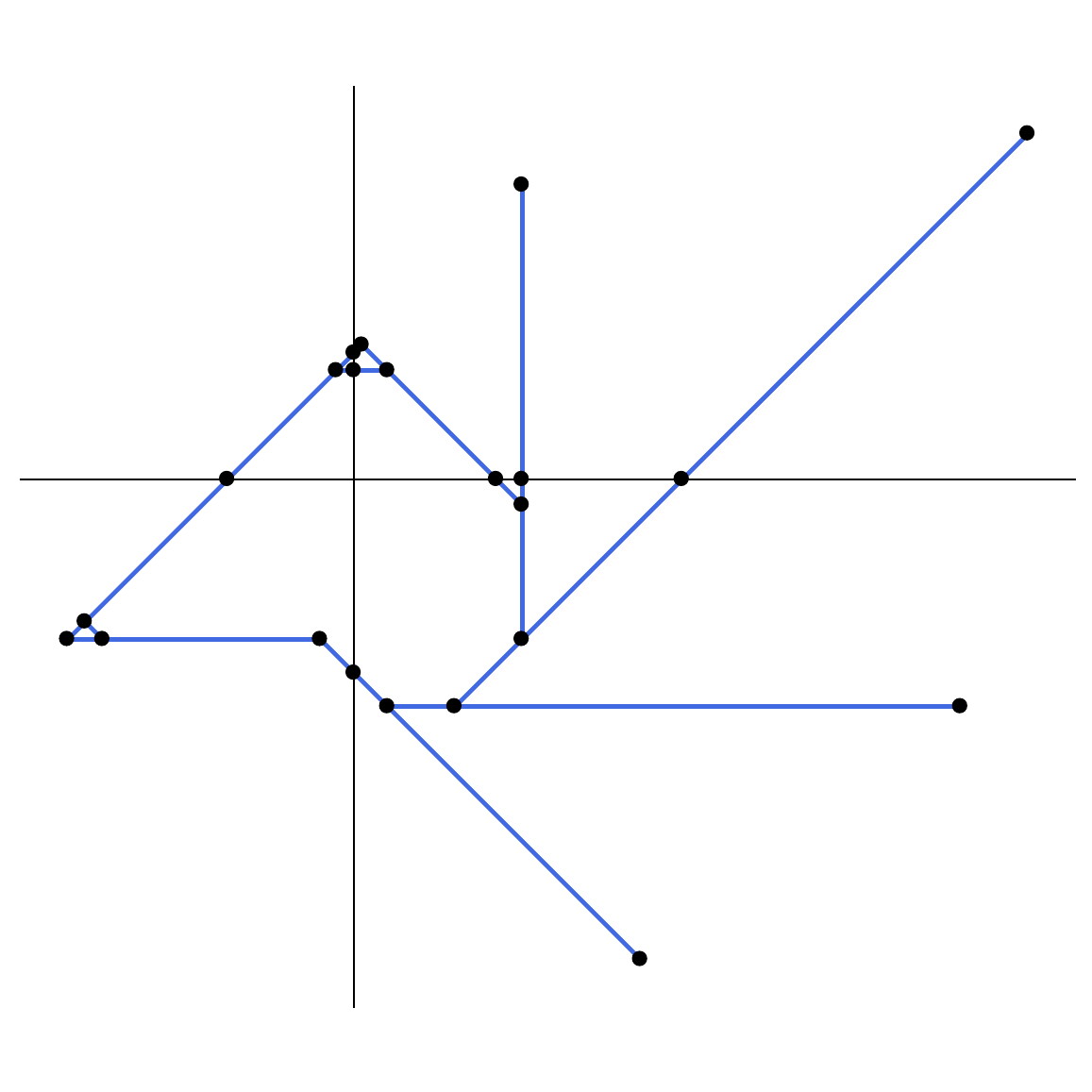}\hfill
\includegraphics[width=0.31\textwidth]{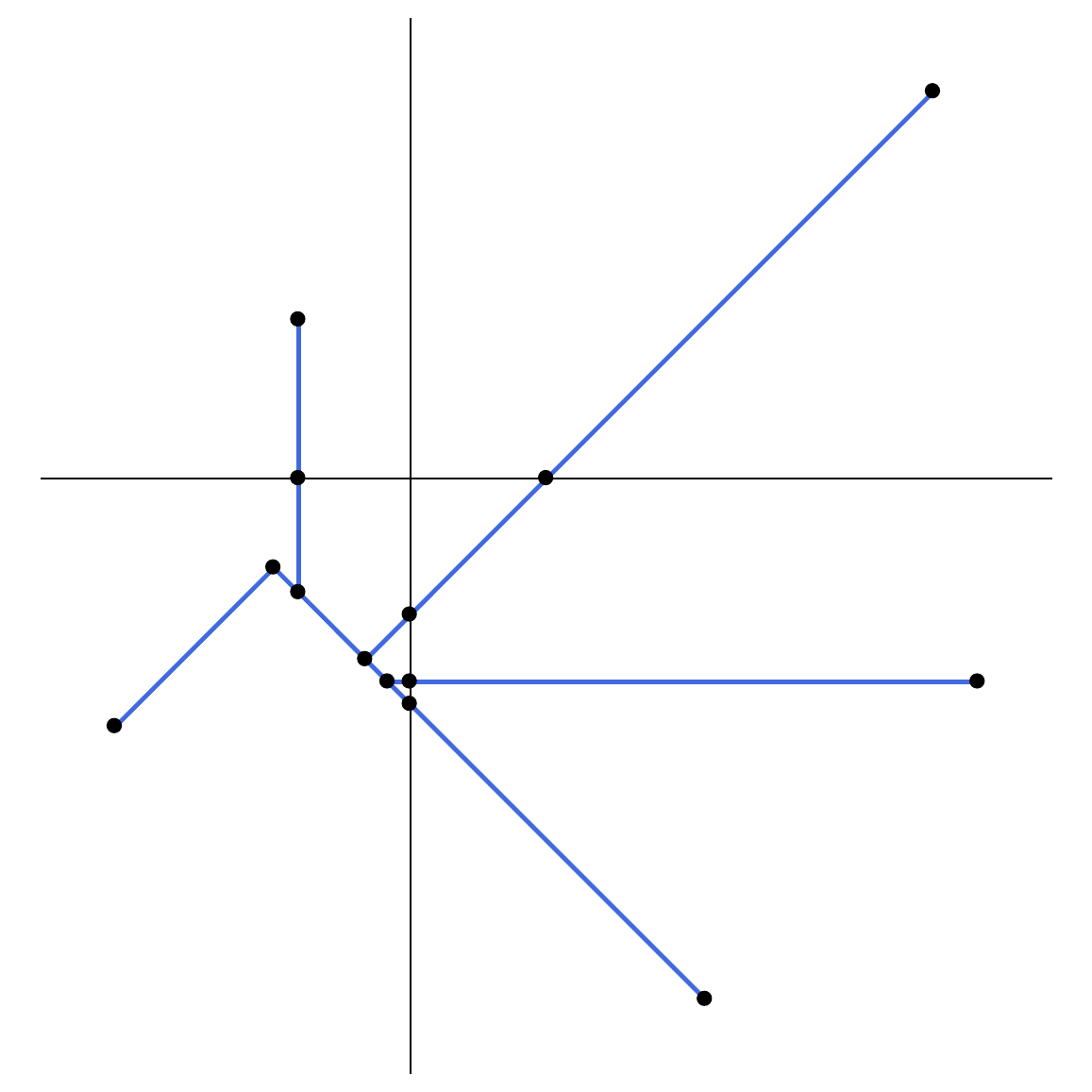}

\vspace{0.5cm}

\includegraphics[width=0.31\textwidth]{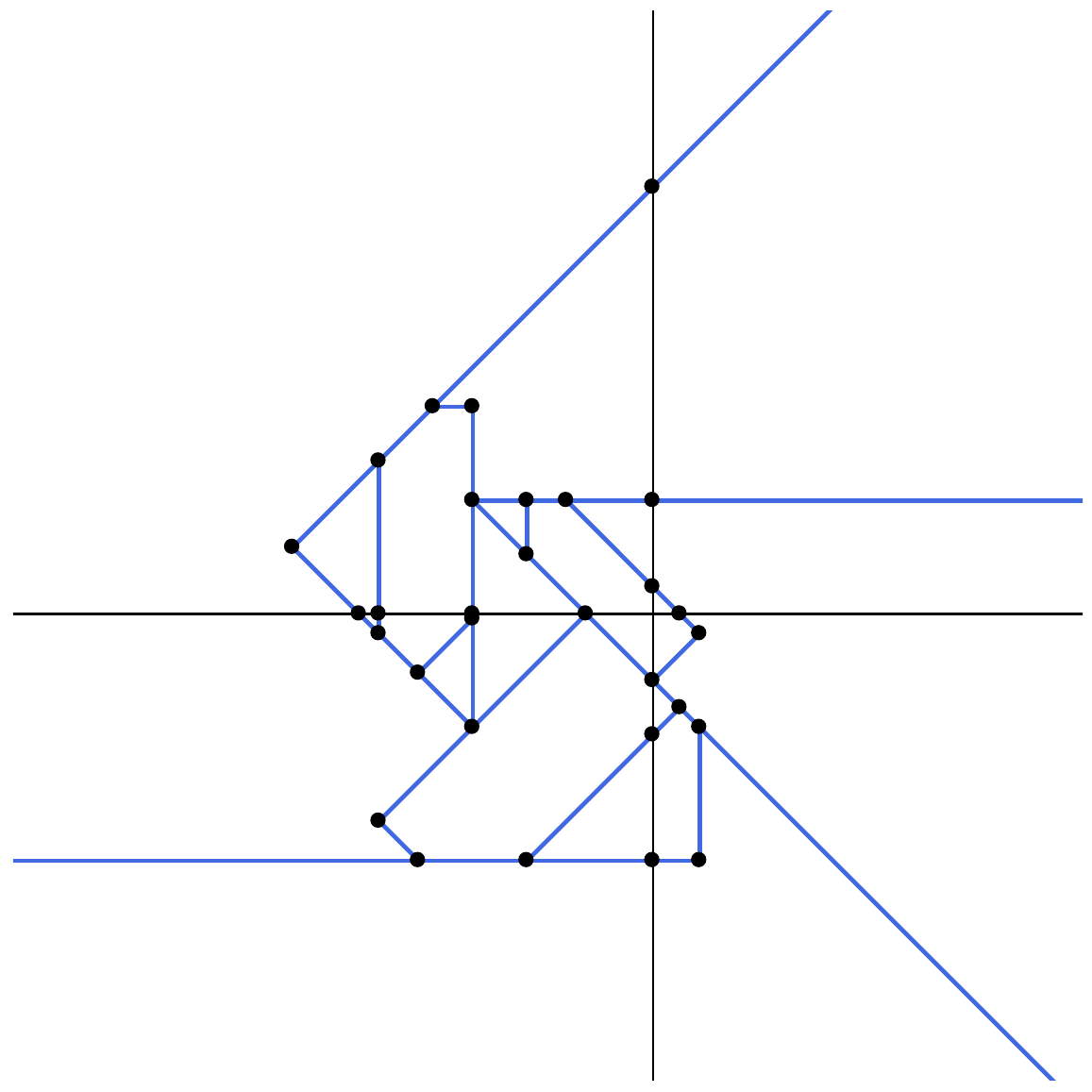}\hfill
\includegraphics[width=0.31\textwidth]{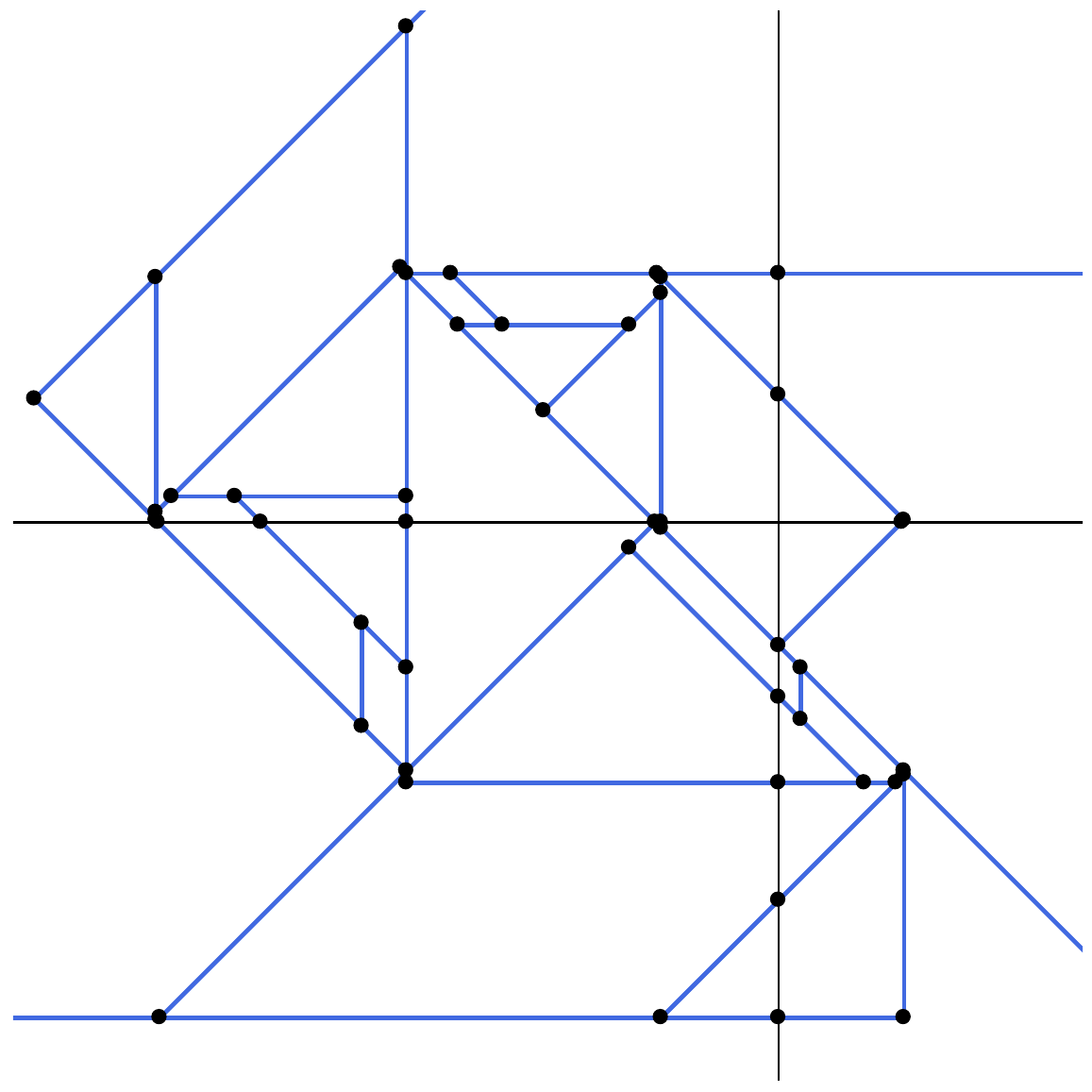}\hfill
\includegraphics[width=0.31\textwidth]{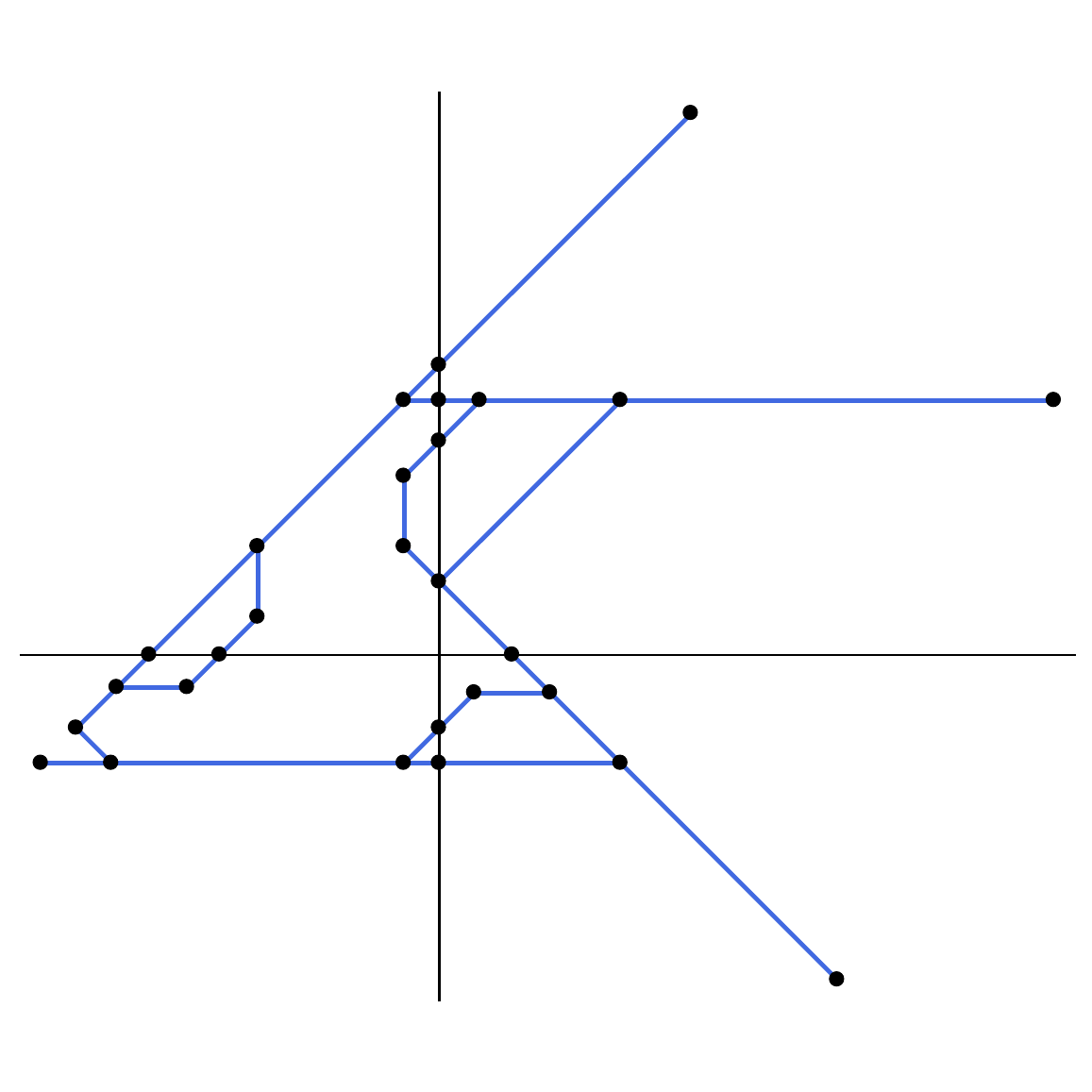}

\caption{Six different configurations of the invariant graphs for $F$ among the 37 presented in \cite{CGMM}. The dots in the graph indicate the vertices and the intersections with the coordinate axes; they have no significance in this figure.}
\label{fig:graphs}
\end{figure}

From Theorem \ref{t:teoB} we get that all the $\omega$-limit sets of $F$, except some fixed points and periodic orbits that appear for some values of $b$ and which are explicitly given, are contained in $\Gamma$. Thus, to study the dynamics of $F$ when $a<0$ it suffices to study the dynamics of $F\vert_{\Gamma}.$   As a first result,  we demonstrated the existence of an open and dense set formed by the pre-images of the segments within each graph $\Gamma$ that collapse to a point (that we call  \emph{plateaus}, see Section \ref{s:s2}), obtaining the following result which characterize that the number of possible $\omega$-limit sets of the pre-images of these plateaus is at most three. Depending on the values of the parameters, these three $\omega$-limit sets can be periodic orbits, Cantor sets or other more complicated subsets of~$\Gamma.$

 \begin{teoa}\label{t:teoC} Set $a<0,$ $b\in\R$ and  let $\Gamma$ be  the corresponding invariant graph for $F$ given in Theorem \ref{t:teoB}. Then for  an open and dense set of  initial conditions $\mathbf{x}\in \Gamma$    there are  at most three possible   $\omega$-limit sets.  Moreover, if  $b/a\in \Q$ these  $\omega$-limit sets are periodic orbits.
\end{teoa}

To study the dynamics on each graph $\Gamma$ we define a suitable partition of the graph and consider the oriented graph associated with the partition. This approach allows us to study the \emph{topological entropy}, $h_F$,  of $F\vert_{\Gamma}$.  Roughly speaking, the entropy $h_F$, introduced and developed in the seminal papers \cite{AKM,Bo3,MZ}, is a non-negative real number associated with a map, measuring its combinatorial complexity. In the next section, we recall its definition in more detail.   In particular, notice that if $h_F = 0$, the dynamics is considered simple, although some complicated limit sets, as Cantor ones, can appear for instance when the graph is a topological circle, the map has degree one and $F\vert_{\Gamma}$ has irrational rotation number. When  $h_F > 0$, the dynamics is really much more complex. Specifically, when $F\vert_{\Gamma}$ has positive entropy ($h_F > 0$), it exhibits periodic orbits with infinitely many distinct periods, and the orbits display various combinatorial behaviors. In particular, it can be shown that if $F|_{\Gamma}$ has positive entropy, then it is chaotic in the sense of Li and Yorke, which implies, among other things, that it has periodic points with arbitrarily large periods and that there exists an uncountable set $\mathcal{S} \subset \Gamma$, the \emph{scrambled set,} such that for any $p, q \in \mathcal{S}$ and each periodic point $r$ of $F\vert_{\Gamma}$, the following holds:
\[
\limsup_{n\to\infty} |F^n(p) - F^n(q)| > 0, \, \liminf_{n\to\infty} |F^n(p) - F^n(q)| = 0, \mbox{ and } 
\limsup_{n\to\infty} |F^n(p) - F^n(r)| > 0.
\]


Our findings are summarized in the following theorem:

\begin{teoa}\label{t:teoD}

 Set $a<0, b\in\R$ and define $c=-b/a$. Consider the map $F$ given in~\eqref{e:F}, restricted to its corresponding invariant graph $\Gamma$ given in Theorem~\ref{t:teoB}. Then, there exist $\alpha$ and $\beta$ such that $F|_{\Gamma}$  has positive entropy if and only if $c\in (\alpha,-1/36)\cup(\beta,1)\cup(1,8),$ where $\alpha\in (-112/137,-13/16)\approx(-0.8175,-0.8125),$  $\beta\in(603/874,563/816)\approx(0.6899,0.6900)$ and in these two intervals the entropy of $F\vert_{\Gamma}$ is non-decreasing in $c.$ Moreover,
	the entropy as a function of $c$ is discontinuous at $c=-1/36$.

\end{teoa}

\subsection{Main results}

We present some further results related to those mentioned earlier. These results aim to provide a deeper understanding of the subtleties of the dynamics of $F$, as well as to illustrate in more detail the techniques we employ. Throughout this paper, we assume that $a<0$, which corresponds to the case in which invariant graphs with complex internal dynamics arise.

It is well known that the entropy as a function defined in the space of continuous self maps of the interval is
lower semi-continuous (see \cite{ALM}). In this book there is a chapter collecting results about continuity properties of the entropy defined in different spaces. For example it is continuous when  we consider the entropy defined in the space of ${\cal C}^{\infty}$ selfmaps of the interval. Also for unimodal maps it is continuous on all functions with positive entropy. 

Inspired by this last result it seems natural to investigate the continuity of the entropy of our family in the values $c$ in which there is a transition from zero  to positive entropy. These values are $\alpha,-1/36,\beta,1$ and $8.$ In $\cite{CGMM}$ we already noted that the entropy is discontinuous at $c=-1/36,$ see Theorem~\ref{t:teoD}. Also, since in a neighborhood of $\alpha$ and $\beta$ the family is conjugated of a family previously studied in \cite{BMT}, we can assert that the entropy is continuous at $c=\alpha$ and $c=\beta.$

 In Section \ref{s:entropyanalysis} we give a detailed analysis of the entropy behavior for $4<c<8$, which in particular demonstrates the continuity of the entropy at $c=8$. Our main results are given in next theorem  and in Proposition~\ref{p:teo1}. We remark that a similar approach will be enough to prove the continuity of the entropy at $c=1.$  Beyond these parameter values, we do not obtain continuity results. Our analysis focuses on parameter regions where the entropy changes from zero to positive values. Extending these results is challenging, as any general approach must take into account changes in the topology of the invariant graphs.

\begin{teo}\label{t:teocont}
 Assume that $a<0$. The entropy function $h_F(c)$ is continuous at the parameter value $c=-b/a=8$.
\end{teo}

The above theorem, together with the results of~\cite{CGMM}, induce to think about some natural questions that we do not face in this work: Which are the regularity properties of  $h_F(c)$ as a function of $c$? Is $c=-1/36$ the unique discontinuity point of  $h_F(c)$?  Which are the values of $c$ that provide the maximum entropy in our family of maps?

The values of the parameters $\alpha$ and $\beta$ in the statement of Theorem \ref{t:teoD} have only been given with few significant digits. In Sections \ref{ss:alpha} and \ref{ss:beta} we propose a constructive way to obtain rational upper and lower bounds for them, as sharp as desired. Although the main contribution is the methodology itself, as an application of our method we prove:

\begin{propo}\label{pro:ab} Let $\alpha$ and $\beta$ be the two values appearing in Theorem~\ref{t:teoD}. Then
$$\alpha\in\left(-\frac{1049417824596806956103568}{1284474531463219438945271},-\frac{140850476140085945702816746162288}{172399253286857828660669132569609}\right).$$ The length of this interval is less than $4\times 10^{-40}$. And
$$\beta\in\left(\frac{945506314303393205598153}{1370433212950874384162254}, \frac {798396920638883099973166531706985228123}{	1157210312199077596904301690272087447914}\right).$$  The length of this interval is less than $5\times 10^{-50}$.
\end{propo}	
In practical terms, the above result implies that
\begin{align*}
		\alpha &= -0.817001660127394075579379106922368833240\ldots,\\
		\beta & = 0.68993242820457428670048891295078173870526\ldots,
	\end{align*}	
	where all shown digits are correct.

As we already comment in  \cite{CGMM}, we suspect that the statement of Theorem \ref{t:teoC} is satisfied by a set of full Lebesgue measure in $\Gamma$. The validity of this fact could explain why numerical simulations only reveal periodic orbits.  We refer the reader to the discussion at the beginning of Section~\ref{s:full}. In the earlier paper, we did not provide any proof to support this suspicion. In the present paper, we prove that, for certain parameter values, the set of full-measure orbits converges to at most three $\omega$-limit sets, which are periodic orbits whenever the parameters are rational.
More concretely, the result is a straightforward consequence of   Propositions~\ref{p:propofull} and~\ref{p:x} proved in Section~\ref{s:full}.

 \begin{teo}\label{t:teonou} Set $a<0,$ $b\in\R$,  $c=-b/a$ and  let $\Gamma$ be  the corresponding invariant graph for $F$ given in Theorem \ref{t:teoB}. Then for  $c<-2$ or for 
 	$c\in [-112/137,-13/16]\cup [603/874,563/816]$  and
 	for a full Lebesgue measure of  initial conditions in~$\Gamma$, there are  at most three possible   $\omega$-limit sets.  Moreover, if  $b/a\in \Q$ these  $\omega$-limit sets are periodic orbits.
\end{teo}

 Clearly, this is only a partial result, as it applies to only a subset of the parameter space. At present, we do not have a unified approach that would allow us to address the problem without relying on the specific dynamical properties of $F$ on particular invariant graphs.

Again, this result raises new questions that we find interesting. For example, what kind of geometry do the basins of attraction in $\mathbb{R}^2$ of the $\omega$-limit sets in the above result exhibit? Are they dense in the plane? These questions are beyond the scope of the present work and we do not address them here.

 Finally, we would like to point out that beyond the results obtained for this family of maps, in the paper we develop methodological tools for studying entropy transitions through invariant graphs, give rational bounds for bifurcation parameters and proving the full measure of certain attractors within the graphs.

\section{Preliminary results}

In this section we recall some basic concepts about entropy of one dimensional maps and also introduce some notations used in  our work. 
Again, the literature on the subject is vast. 
As illustrative references only, we refer the reader to the pioneering works \cite{AKM,Bo3,MZ}, the reference book \cite{ALM}, the short survey \cite{Ll15} and the references therein, as well as the following papers dealing specifically with topological entropy on graphs and trees: \cite{ABK,AJM,BGMY,LlM93}, cited here without any claim of exhaustiveness.

\subsection{Topological entropy}\label{a:prelim}

We present the definition of topological entropy for the particular class of maps that we will use throughout this work. In \cite{ALM}, more precise definitions and results can be found for interval or circle maps, which are adaptable for maps of compact graphs.

We say that $F$, a self-map of a compact graph $\Gamma$,  is {\em piecewise monotone} if there exists a finite covering $\mathcal{A}$ of $\Gamma$ by intervals (i.e., segments of edges) such that the image of each interval in $\mathcal{A}$ is an interval and $F$, restricted to each interval $A$, is continuous and monotone for all $A\in \mathcal{A}$.  \emph{From now on, we consider that  $F:\Gamma\longrightarrow \Gamma$ is a piecewise monotone map of the compact graph $\Gamma$. }

Let $\mathcal{P}=\{I_1,\ldots,I_n\}$ be a finite partition of $\Gamma$ by closed intervals. We say that $\mathcal{P}$ is a {\em monopartition} if $F(I_i)$ is an interval and $F\vert_{I_i}$ is continuous and monotone for all $i\in\{1,\ldots,n\}.$ We call the endpoints of the intervals $I_i$ the \emph{turning points}, and we denote the set of all such points by $C.$ When $x\in \Gamma\setminus C$, we define the \emph{address} of $x$ as $A(x):=I_i$ if $x\in I_i.$ If $x\in \Gamma\setminus\left( \bigcup_{i=0}^{m-1}F^{-i}(C)\right)$, we define the {\em itinerary of length $m$ of $x$} as the sequence of symbols
$I_m(x)=A(x)A(F(x))\ldots A(F^{m-1}(x)).$

Let $N(F,\mathcal{P},m)$ be the number of distinct itineraries of length $m$ (note that $N(F,\mathcal{P},m)\leq n^m$). Then, the topological entropy can be defined using the following result:

\begin{lem}\label{grownumber} 
	Let $F:\Gamma\longrightarrow \Gamma$ be a piecewise monotone map on a compact graph $G$. Let $\mathcal{P}$ be a monopartition. Then the limit 
	$s(F):=
	\lim\limits_{m\to \infty} \sqrt[m]{N(F,\mathcal{P},m)}$, called  \emph{the growth number}, exists. Moreover, it is independent of the choice of the monopartition $\mathcal{P}$, and 
	\[
	h(F):=\ln\left(\lim\limits_{m\to \infty} \sqrt[m]{N(F,\mathcal{P},m)}\right)
	\]
	is the \emph{topological entropy} of $F.$
\end{lem}

It is well-known that the entropy is an invariant for conjugation. In the case of interval maps, it is also an invariant by the so called {\em semiconjugation.} Recall that two piecewise monotone self maps on the interval $I,$  $f$ and $g,$ are said  to be {\em semiconjugated} if there exists a non-decreasing map $s:I\longrightarrow I$ such that $s(I)=I$ and $g\circ s=s\circ f.$

The entropy can be calculated in a straightforward way when the graphs are Markov graphs. We say that $\mathcal{P}$ is a {\em Markov partition} if, for all $I\in\mathcal{P},$ $F(I)$ is the union of some elements of $\mathcal{P}.$ Clearly, in this case, the set of turning points is invariant.

\emph{From now on, we consider that $\mathcal{P}=\{I_1,\ldots,I_n\}$ is a monopartition of $\Gamma$.} We consider {\em the matrix $M(F,\mathcal{P})$  associated with $\mathcal{P}$} as the $n\times n$-matrix defined by
\[
m_{i,j}=\left\{
\begin{array}{ll}
	1, & \mbox{if } I_j\subset F(I_i); \\
	0, & \mbox{otherwise.}
\end{array}
\right.
\]
We denote the spectral radius on $M(F,\mathcal{P})$, i.e. the largest absolute value of its eigenvalues, by $r(\mathcal{P})$. From the Perron-Frobenius Theorem, we know that the spectral radius of $M(F,\mathcal{P})$ is achieved by a positive real eigenvalue. The following result allows us to calculate the entropy using the Markov matrix:

\begin{lem}\label{Markov} 
	Let $F:\Gamma\longrightarrow \Gamma$ be a piecewise monotone map on a compact graph, and let $\mathcal{P}=\{I_1,\ldots,I_n\}$ be a monopartition. Then 
	$
	r(\mathcal{P})\leq s(F).
	$
	Moreover, if $\mathcal{P}$ is a Markov partition, then $r(\mathcal{P})=s(F).$
\end{lem}

Therefore, for Markov matrices $M(F,\mathcal{P})$:
\[
h(F)=\ln\left(r(\mathcal{P})\right),
\]
where $r(\mathcal{P})$ is the spectral radius of the matrix.

To compute the spectral radius of the above mentioned Markov matrices, 
in Section \ref{s:entropyanalysis} we have used, repeatedly, the method developed in \cite{BGMY}. We briefly recall it. We construct an abstract oriented graph whose vertices are $I_1, \ldots, I_n$, and there is a directed edge from $I_i$ to $I_j$ if and only if $I_j \subset F(I_i).$ Next, we introduce the notion of a \emph{rome}.

\begin{defi}\label{d:roma}
	Let $M=(m_{ij})_{i,j=1}^n$ be an $n \times n$ matrix with $m_{i,j} \in \{0,1\}$. For a sequence $p=(p_j)_{j=0}^k$ of elements in $\{1,2,\ldots, n\}$, its width $w(p)$ is defined as $w(p) = \prod_{j=1}^k m_{p_{j-1}p_j}.$ The sequence $p$ is called a path if $w(p) \neq 0.$ In this case, $k = \mathit{l}(p)$ is the length of the path $p.$ A subset $R \subset \{1,2, \ldots, n\}$ is called a rome if there is no loop outside $R$, i.e., there is no path $(p_j)_{j=0}^k$ such that $p_0 = p_k$ and $(p_j)_{j=0}^k$ is disjoint from $R.$ For a rome $R$, a path $(p_j)_{j=0}^k$ is called simple if $\{p_0, p_k\} \subset R$ and $\{p_1, \ldots, p_{k-1}\}$ is disjoint from $R.$
\end{defi}

Informally, \emph{a rome} $R$ is a collection of vertices in the oriented graph such that any loop in the graph must pass through an element of $R$. Note that a path in the matrix associated with the oriented graph corresponds to a path in the graph itself. For a rome $R = \{r_1, \ldots, r_k\}$, where $r_i \neq r_j$ for $i \neq j$, we define a matrix function $A_R$ by $A_R = (a_{ij})_{i,j=1}^k$, where
\[
a_{ij}(\lambda) = \sum_{p} w(p) \, \lambda^{-\mathit{l}(p)},
\]
and the summation is taken over all simple paths originating at $r_i$ and terminating at $r_j$. Let $E$ denote the identity matrix (of the appropriate size).

\begin{teo}\label{rome} (See \cite{BGMY})
	Let $R = \{r_1, \ldots, r_k\}$ (with $r_i \neq r_j$ for $i \neq j$) be a rome. Then the characteristic polynomial of $M$ is given by 
	$
	(-1)^{n-k} \, \lambda^n \, \mathrm{det}(A_R(\lambda) - E).
	$
\end{teo}

The next lemma is used in Sections \ref{ss:alpha} and \ref{ss:beta}.

\begin{lem}\label{entroo}  If $f:X\longrightarrow X$ is a continuous map on a compact space, then $h(f^n) = n h(f).$
\end{lem}

\subsection{A useful conjugation and an observation}\label{s:s2}

The family of maps $F_{a,b}$ has only one essential parameter, since for any $\lambda>0,$  \begin{equation}\label{conj}\lambda F_{a,b}(x/\lambda,y/\lambda)=F_{\lambda a,\lambda b}(x,y).\end{equation} 
This equality implies that for any $a,b\in \R^2$ and for any $\lambda>0$ the maps $F_{\lambda a,\lambda b}$ and $F_{a,b}$ are conjugate. Hence in our proofs we can restrict our attention to three cases $a\in\{-1,0,1\}.$

In consequence, to study the case $a<0$ it suffices to consider the case $a=-1$ by using the conjugation \eqref{conj}. The conjugation can be used afterwards to obtain the results for $F_{a,b}$ from its corresponding map $F_{-1,-b/a}.$

We end this section with a preliminary observation. 
For $i=1,2,3,4,$ we denote by $F_i$ the expression of the affine map $F$ restricted to each one of the quadrants $Q_1=\{(x,y): x\geq 0, y\geq 0\},Q_2=\{(x,y): x\leq 0, y\geq 0\},Q_3=\{(x,y): x\leq 0, y\leq 0\}$ and $Q_4=\{(x,y): x\geq 0, y\leq 0\},$ 
that is,
\begin{equation}\label{e:Fis}
\begin{array}{ll}
	F_1(x,y)=(x-y+a,x-y+b),\,&F_2(x,y)=(-x-y+a,x-y+b),\\
	F_3(x,y)=(-x-y+a,x+y+b),\,&
	F_4(x,y)=(x-y+a,x+y+b).
	\end{array}
\end{equation}
We see that the straight lines of slope $1$ contained in $Q_1$ and also the straight lines of slope $-1$ contained in $Q_3$ collapse to a point. For reasons that are explained in \cite{CGMM}, which are related to the terminology in  \cite{BMT}, we call these intervals \emph{plateaus}. Hence,  when calculating the entropy of the maps $F|_\Gamma$, where $\Gamma$ is the graph that appears in Theorems \ref{t:teoB} and \ref{t:teoD}, the intervals of the associated auxiliary oriented graph which represent the preimages of these points  will not be considered.

\section{Proof of Theorem~\ref{t:teocont}}\label{s:entropyanalysis}

By using the conjugation \eqref{conj}, in order to study the case $a<0$ and $4<c<8$, we can take  $a=-1$ and  $4<b<8$. In this section we give a detailed analysis of entropy behavior for this range of the parameters. According to Theorem \ref{t:teoD}, the map has positive entropy for this range of parameters, and is zero for $b\geq 8$.  The main result of this section is  Theorem~\ref{t:teocont}.

This theorem is a consequence of 
Proposition~\ref{p:teo1}, given below, which either characterizes or bounds the entropy $F|_{\Gamma}$ in a certain partition of the parameter's interval $(4,8)$, giving a more accurate description than the one in Theorem \ref{t:teoD}. In particular,  this proposition shows that $\lim\limits_{b\to 8^-} h_F(b)=0$. Since, from  Theorem \ref{t:teoD}, $h_F(b)=0$ for $b\geq 8$, the result follows.

 This section will be devoted, almost entirely, to demonstrating Proposition~\ref{p:teo1}. To state the proposition we consider the following partition of the parameter's interval: 
$$(4,8)=\bigcup_{n=0}^{\infty} (p_{n-1},p_n]=\bigcup_{n=0}^{\infty}\left(S_n\cup T_n\cup U_n\cup V_n\right),$$
with
$$S_n=(p_{n-1},s_n),\,T_n=[s_n,r_n],\,U_n=(r_n,q_n),\,V_n=[q_n,p_n],$$
where 
\begin{equation}\label{e:rec}
p_n=\frac{4(4\cdot 4^{n+1}-1)}{2\cdot 4^{n+1}+1}\,\,\,,\,q_n=\frac{8\cdot 4^{n+1}+1}{4^{n+1}+2}\,,\,\,\,r_n=\frac{16\cdot 4^{n+1}-1}{2\cdot 4^{n+1}+4}\,,\,\,\,s_n=\frac{2(4\cdot 4^{n+2}-1)}{4^{n+2}+11},
\end{equation}
and $p_{-1}=4$. 
Notice that that for any $n\in\N\cup\{0\}$:
$p_{n-1}\,<\,s_n\,<\,r_n\,<\,q_n\,<\,p_n.$

 The following result improves Proposition 39 of \cite{CGMM}, where only the positivity of the entropy is established for these parameter values.
\begin{propo}\label{p:teo1}
	Assume that $a=-1$ and $4<b<8$ and let $h_F(b)$ be the entropy of the map $F|_{\Gamma}$. Then the entropy is always positive and,  furthermore, the following holds:
	\begin{enumerate}[(a)]
		\item  If $b\in S_n$, then  $\ln(\alpha_n)\leq h_F(b)\leq \ln(\beta_n),$ where $\alpha_n$  and $\beta_n$ are the unique positive roots of the polynomials
		$P_{\alpha,n}(\lambda)=\lambda^{7+3n}-\lambda^{4+3n}-1$
and  $P_{\beta,n}(\lambda)=\lambda^{7+3n}-\lambda^{4+3n}-\lambda^3-2$ respectively.
		\item  If $b\in T_n$, then  $h_F(b)=\ln(\delta_n),$ where $\delta_n$  is the unique positive root of the polynomial $P_{\delta,n}(\lambda)=\lambda^{7+3n}-\lambda^{4+3n}-2.$
		\item  If $b\in U_n$, then  $\ln(\alpha_n)\leq h_F(b)\leq \ln(\gamma_n)$, where $\gamma_n$ is the unique positive root of   $P_{\gamma,n}(\lambda)=\lambda^{10+3n}-\lambda^{7+3n}-2\lambda^3-1$. 
		\item  If $b\in V_n$, then  $h_F(b)=\ln(\varphi_n),$ where $\varphi_n$ is the unique positive root of  $P_{\varphi,n}(\lambda)=\lambda^{10+3n}-\lambda^{7+3n}-\lambda^3-1$. 
		\end{enumerate}	
		Furthermore, 
\begin{equation}\label{e:relpositions}
1<\alpha_n<\varphi_n<\delta_n<\gamma_n<\beta_n,
\end{equation}
		and each one the five sequences $\alpha_n,\beta_n,\gamma_n,\delta_n, \varphi_n$ is decreasing and tends to $1$ as $n$ tends to infinity, hence the entropy of $F|_{\Gamma}$ when $b$ belongs to the parameter's sets $S_n$, $T_n$ $U_n$ and $V_n$
		tends to $0$ as $n$ tends to infinity (see Figure \ref{f:fig2}). In consequence,
$$\lim\limits_{b\to 8^-} h_F(b)=0.$$
\end{propo}

The first cases given by Proposition \ref{p:teo1} are summarized in Table \ref{tab:taula 1} and Figure \ref{f:fig2}.

\begin{table}[ht]
\centering
\begin{tabular}{|c|c|c|}
\hline
{Set} & {$b$} & Exact entropy or bounds\\ \hline
$S_0$ & $\left(4,\frac{14}{3}\right)$ & $\left[0.14717,0.28888\right]$ \\ \hline
$T_0$ & $\left[\frac{14}{3},\frac{21}{4}\right]$ & $0.20844$ \\ \hline
$U_0$ & $\left(\frac{21}{4},\frac{11}{2}\right)$ & $\left[0.14717,023031\right]$ \\ \hline
$V_0$ & $\left[\frac{11}{2},\frac{20}{3}\right]$ & $0.18600$ \\ \hline

$S_1$ & $\left(\frac{20}{3},\frac{34}{5}\right)$ & $\left[0.11977,0.21132\right]$ \\ \hline
$T_1$ & $\left[\frac{34}{5},\frac{85}{12}\right]$ & $ 0.16389$ \\ \hline
$U_1$ & $\left(\frac{85}{12},\frac{43}{6}\right)$ & $\left[0.11977,0.18155\right]$ \\ \hline
$V_1$ & $\left[\frac{43}{6},\frac{84}{11}\right]$ & $0.15051$ \\ \hline

$S_2$ & $\left(\frac{84}{11},\frac{682}{89}\right)$ & $\left[0.10238,0.17042\right]$ \\ \hline
$T_2$ & $\left[\frac{682}{89},\frac{31}{4}\right]$ & $0.13698$ \\ \hline
$U_2$ & $\left(\frac{31}{4},\frac{171}{22}\right)$ & $\left[0.10238,0.15186\right]$ \\ \hline
$V_2$ & $\left[\frac{171}{22},\frac{340}{43}\right]$ & $0.12795$ \\ \hline
\end{tabular}
\caption{First sets of the partition of the parameter interval $(4,8)$ and exact value (with five significant digits) or bounds of the entropy according to Proposition \ref{p:teo1}.}
\label{tab:taula 1}
\end{table}

\begin{figure}[h]
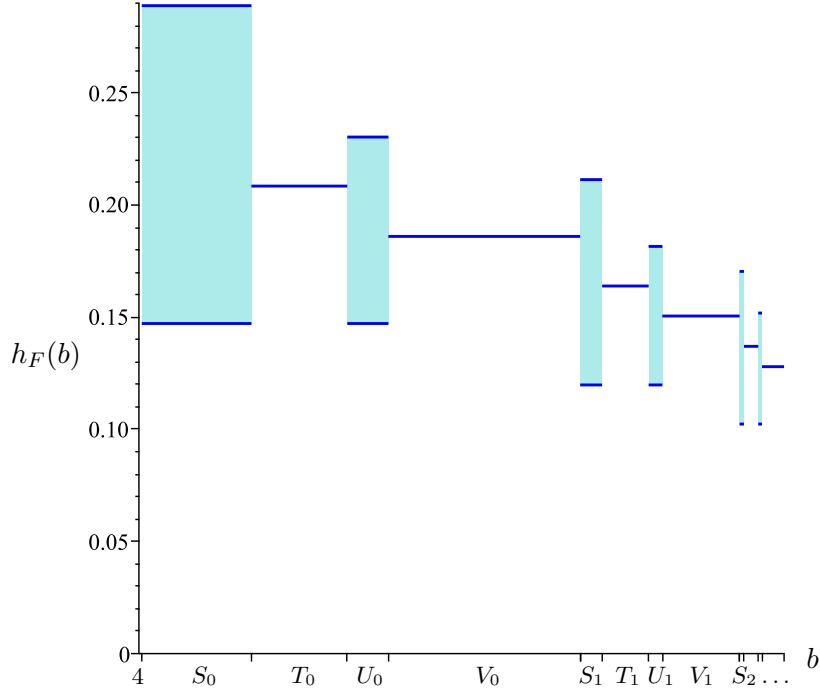

	\centering
	
	\begin{lpic}[l(2mm),r(2mm),t(2mm),b(2mm)]{gr2b2(0.49)}
		
		\lbl[c]{-8,97; $h_F(b)$}
		\lbl[c]{200,15; $b$}
		
		\lbl[c]{17,10; {\footnotesize $4$}}
		\lbl[c]{35,10; {\footnotesize $S_0$}}
		\lbl[c]{62,10; {\footnotesize $T_0$}}
		\lbl[c]{80,10; {\footnotesize $U_0$}}
		\lbl[c]{112,10; {\footnotesize $V_0$}}
		
		\lbl[c]{140,10; {\footnotesize $S_1$}}
		\lbl[c]{150,10; {\footnotesize $T_1$}}
		\lbl[c]{159,10; {\footnotesize $U_1$}}
		\lbl[c]{170,10; {\footnotesize $V_1$}}
		
		\lbl[l]{178,10; {\footnotesize $S_2\ldots$}}
		
	\end{lpic}
	
	\caption{First sets of the partition of  $(4,8)$, and exact value or bounds of the entropy according to Proposition \ref{p:teo1}. In light blue, the bounding interval in the sets $S_i$ and $U_i$.}\label{f:fig2}
\end{figure}

\subsection{Invariant graph $\Gamma$ and a first glimpse on the dynamics of $F|_{\Gamma}$}\label{ss:prelimgamma}

Set $a=-1$ and $4<b<8$. From Theorem \ref{t:teoB} and \cite[Table 1]{CGMM} we know that there exists a graph $\Gamma$, invariant by $F$, that for all $(x,y)\in\mathbb{R}^2,$ $F^5(x,y)\in\Gamma$. The graph $\Gamma$ which captures the dynamics of 
$F$ for this range of parameters is given in Figure~\ref{f:esp3}. The points appearing there are: $P_{1}=(0, b +1)$, $P_{2}=(-b -1, 0)$, $P_{3}=(-b -2, -1)$, $P_{4}=(0, -1)$, $P_{5}=(b, -1)$, $P_{6}=(b +2, -3)$, $P_{7}=(
b -1, 0)$, $P_{8}=(0, b -1)$, $P_{9}=(b +4, 2 b -1)$, $P_{10}=(b, 2 b -1)$, $P_{11}=(b -2, 2 b -1)$, $P_{12}=(
-b +4, 5)$, $P_{13}=(-b, 1)$, $P_{14}=(b -2, -1)$, $P_{15}=(b -2, 2 b -3)$, $P_{16}=\left(b -1, 2 b -1\right)$,  $P_{17}=\left(
-\frac{b}{2}-1, \frac{b}{2}\right)$, $P_{18}=\left(\frac{3 b}{2}-1, 2 b -1
\right)$,  $X_1=\left(-\frac{b}{4}+\frac{1}{2}, -1\right)$,  $X_2=\left(-\frac{b}{2}
, -1\right)$,  $X_3=\left(-b +1, -1\right)$,  $X_4=\left(-b +\frac{1}{2}, -1\right)$, $X_5=\left(-b, -1\right)$, 
$X_6=\left(-\frac{5 b}{4}+\frac{1}{2}, -1\right)$,  $Y_1=\left(
\frac{b}{4}-\frac{1}{2}, \frac{3 b}{4}-\frac{1}{2}\right)$,  $Y_2=\left(
\frac{b}{2}, \frac{b}{2}-1\right)$,  $Y_4=\left(b -\frac{1}{2}, -{\frac{1}{2}
}\right)$,  $Y_6=\left(\frac{5 b}{4}-\frac{1}{2}, -\frac{b}{4}-\frac{1}{2}
\right)$.

	 \begin{figure}[ht]
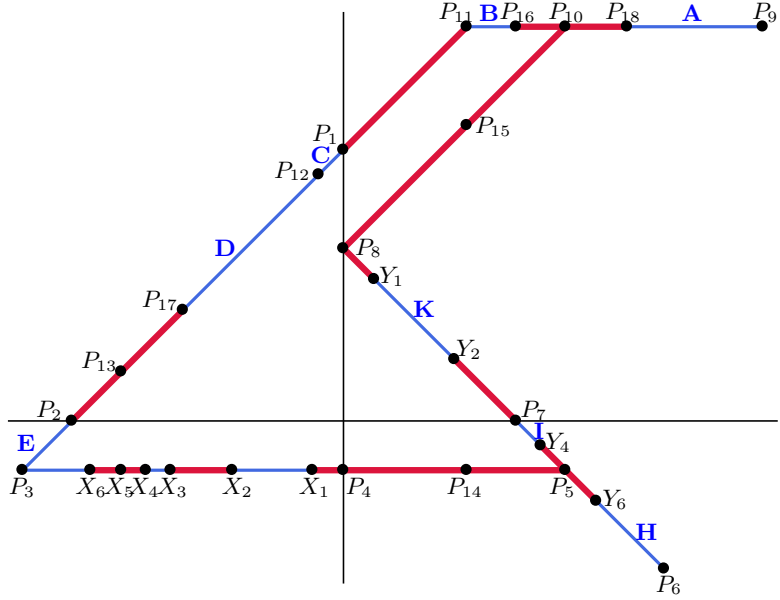
\label{graficespecial3}
	 	\footnotesize
	 	\centering
	 	
	 	\begin{lpic}[l(2mm),r(2mm),t(2mm),b(2mm)]{am-especial-3b(0.54)}
	 		\lbl[l]{3,185; $\boxed{a=-1,\,4< b<8}$}

	 		\lbl[r]{85,138; $P_1$}
	 		
	 		\lbl[r]{83,133; ${\color{blue}\mathbf{C}}$}

	 		\lbl[r]{17,70; $P_2$}
	 		
	 		\lbl[c]{8,62; ${\color{blue} \mathbf{E}}$}
	 		
	 		\lbl[c]{7,51; $P_3$}
	 		
	 		\lbl[c]{90,51; $P_4$}
	 		
	 		\lbl[c]{79,51; $X_1$}
	 		
	 		\lbl[c]{60,51; $X_2$}
	 		
	 		\lbl[c]{44,51; $X_3$}
	 		
	 		\lbl[c]{37,51; $X_4$}
	 		
	 		\lbl[c]{31,51; $X_5$}
	 		
	 		\lbl[c]{24,51; $X_6$}
	 		
	 		\lbl[c]{140,51; $P_5$}

	 		\lbl[l]{163,27; $P_6$}
	 		
	 		\lbl[l]{89,110; $P_8$}
	 		
	 		\lbl[c]{133,70; $P_7$}

	 		\lbl[c]{190,168;  $P_9$}
	 		
	 		\lbl[c]{141,168; $P_{10}$}
	 		
	 		\lbl[c]{114,168;  $P_{11}$}
	 		
	 		\lbl[c]{122,168;  ${\color{blue} \mathbf{B}}$}
	 		
	 		\lbl[r]{78,129; $P_{12}$}
	 		
	 		\lbl[c]{57,110; ${\color{blue} \mathbf{D}}$}
	 		
	 		\lbl[c]{26,82; $P_{13}$}
	 		
	 		\lbl[c]{116,51; $P_{14}$}
	 		
	 		\lbl[c]{123,140; $P_{15}$}

	 		\lbl[c]{129,168;  $P_{16}$}
	 		
	 		\lbl[c]{41,97;  $P_{17}$}
	 		
	 		\lbl[c]{155,168;  $P_{18}$}
	 		
	 		\lbl[c]{172,168;  ${\color{blue}\mathbf{A}}$}

	 		\lbl[c]{98,103;  $Y_1$}
	 		
	 		\lbl[c]{106,95;  ${\color{blue}\mathbf{K}}$}
	 		
	 		\lbl[c]{117,85;  $Y_2$}
	 		
	 		\lbl[l]{133,65;  ${\color{blue}\mathbf{I}}$}
	 		
	 		\lbl[c]{139,62;  $Y_4$}
	 		
	 		\lbl[c]{153,48;  $Y_6$}
	 		
	 		\lbl[l]{158,40;  ${\color{blue} \mathbf{H}}$}
	 	\end{lpic}
	 	
	 	\caption{The graph $\Gamma$ for $a=-1$ and $4<b<8$.}\label{f:esp3}
	 \end{figure} 
	 
In Figure~\ref{f:esp3}, the points $P_i$ describe the basic elements of the graph: vertices and intersections with the axes. In order to describe the dynamics, we need to introduce some extra points, which are also shown in this figure.

\begin{itemize}
	\item The preimages of $P_1=(0,b+1):$
	$$X_2:=\left(\frac{-b}{2},-1\right)\rightarrow Y_2:=\left(\frac{b}{2},\frac{b-2}{2}\right)\rightarrow P_1.$$
	\item The preimages of $P_2=(-b-1,0):$
	$$X_4:=\left(\frac{1-2b}{2},-1\right)\rightarrow Y_4:=\left(\frac{2b-1}{2},\frac{-1}{2}\right)\rightarrow P_{16}:=(b-1,2b-1)\rightarrow P_2.$$

		\item Concerning the preimages of $P_4=(0,-1),$ we consider $P_{17}:=\left(-\frac{2+b}{2},\frac{b}{2}\right)\rightarrow P_4.$ We have that the straight line $y=x+b/2$ with $x>0$ collapses to $P_{17},$ and it contains two points on the graph, $Y_1:=\left(\frac{b-2}{4},\frac{3b-2}{4}\right)$ and $P_{18}:=\left(\frac{3b-2}{2},2b-1\right).$ Now:
	$$X_1:=\left(\frac{2-b}{4},-1\right)\rightarrow Y_1\rightarrow P_{17}\rightarrow P_4.$$
	 and
	 $$X_6:=\left(\frac{2-5b}{4},-1\right)\rightarrow Y_6:=\left(\frac{5b-2}{4},-\frac{b+2}{4}\right)\rightarrow P_{18}\rightarrow P_{17}\rightarrow P_4.$$

	\item The preimage of $P_5=(b,-1):$ $X_5:=(-b,-1)\rightarrow P_5$.	 
	 
	\item The preimage of $P_7=(b-1,0):$ $X_3:=(1-b,-1)\rightarrow P_7$.
\end{itemize}

	 The intervals that collapse are explicitly displayed in red color in Figure \ref{f:esp3}. We begin by naming: $A:=\overline{P_{9}P_{18}}$, $B:=\overline{P_{11}P_{16}}$, $C:=\overline{P_{1}P_{12}}$, $D:=\overline{P_{12}P_{17}}$, $E:=\overline{P_{2}P_{3}}$, $F_1:=\overline{P_{3}(x_0,-1)}$, $G:=\overline{(x_0,-1)P_{4}}$, $H:=\overline{P_{6}Y_{6}}$, $I:=\overline{Y_{4}P_{7}}$, $K:=\overline{Y_{2}Y_{1}},$ where $(x_0,-1):=F(P_{12})=(b-10,-1)$.
	 
	The images of these intervals can be easily found, except for the intervals $F_1$ and $G.$ For these intervals,  we have to consider five different cases, depending on the location of the point $(x_0,-1)$ defined above.

\textbf{Case 1:} For $4<b<14/3$, $(x_0,-1)\in \overline{P_{3}X_{6}}$. This case corresponds with $b\in S_0=(p_{-1},s_0)$ using the notation in \eqref{e:rec}. In this case, the interval $G$ covers $K,I$ and also a part of $H.$ The interval $F_1$ just covers the other part of $H.$ Hence the corresponding oriented graph is not of Markov type. 
\emph{For this reason we plot dashed arrows departing from $F_1,G$ to $H$ in the oriented graph below}, and we will follow this graphic convention in the following. The technique of considering the \emph{expanded directed graph}, depicted with dashed lines, and calculating its entropy as if it were a Markov graph, allows us to bound the entropy from above.

Hence, by adding a dashed arrow from $G$ to $H$ and another from $F_1$ to $H$, we get:

	\begin{center} \begin{tikzcd}
		&             &             &             & K \arrow[d] \arrow[r] & C \arrow[r]                                                & F_1 \arrow[lllld, dashed, bend right] \\
		B \arrow[r] & E \arrow[r] & H \arrow[r] & A \arrow[r] & D \arrow[r]           & G \arrow[lu] \arrow[llllld] \arrow[lll, dashed, bend left] &                                       \\
		I \arrow[u] &             &             &             &                       &                                                            &                                      
	 \end{tikzcd} \end{center}

We see that this graph has two loops with length $3$ and $7$ respectively, while the extended graph  (the one with dashed lines) has four loops of length $3,4,7$ and $7$. 

We can obtain bounds for the entropy $h_F(b)$, in this case, by using the \emph{romes' method}, introduced in \cite{BGMY} and summarized in Theorem \ref{rome} of the previous section. According to Definition~\ref{d:roma}, the segment $G$ is a rome of both directed graphs. The function $A_R$, in this case, is simply
$$
A_R(\lambda)=\frac{1}{\lambda^3}+\frac{1}{\lambda^7},
$$ hence, using  Theorem \ref{rome}, 
the characteristic polynomial of the Matrix associated with the first directed graph, whose entropy bounds $h_F(b)$ from below, is
$$(-1)^9\lambda^{10}\det(A_R(\lambda)-1)=
(-1)^9\lambda^{10}\left(\frac{1}{\lambda^3}+\frac{1}{\lambda^7}-1\right)=\lambda^3(\lambda^7-\lambda^4-1).
$$
Observe that the useful information of the characteristic polynomial is given by the term $\lambda^7-\lambda^4-1$ which corresponds with the polynomial $P_{\alpha,0}(\lambda)$ defined in the statement (a) of Proposition \ref{p:teo1}. \emph{In this work, abusing notation, we will call the characteristic polynomial only the relevant term of the authentic characteristic polynomial, that is, a term that gives rise to the largest positive root.} Proceeding in the same way, we get that the  function $A_R(\lambda)$ of the directed graph with dashed lines, that helps us to bound from above the entropy, is 
$$
A_R(\lambda)=\frac{1}{\lambda^3}+\frac{1}{\lambda^4}+\frac{2}{\lambda^7},
$$ hence,
the characteristic polynomial is 
$$P_{\beta,0}(\lambda)=\lambda^{7}-\lambda^{4}-\lambda^3 -2.
$$
Using Descartes' rule of signs, both polynomials $P_{\alpha,0}$ and $P_{\beta,0}$ have a unique positive root given by 
$\alpha_0\approx 1.15855$ and $\beta_0\approx 1.33493$, hence for this range of the parameter $b$ we obtain
$$
0.14717\approx \ln(\alpha_0)\leq h_F(b)\leq \ln(\beta_0)\approx 0.28888.
$$

\textbf{Case 2:} For $14/3\leq b\leq 21/4,$ i.e. $b\in T_0=[s_0,r_0]$, then $(x_0,-1)\in \overline{X_{4}X_{6}}.$ Now the interval $G$ covers exactly  $I\cup K$ and $F_1$ covers exactly $H.$ The corresponding oriented graph is now of Markov type:

\begin{center} \begin{tikzcd}
	&             &             &             & K \arrow[d] \arrow[r] & C \arrow[r]                 & F_1 \arrow[lllld, bend right] \\
	B \arrow[r] & E \arrow[r] & H \arrow[r] & A \arrow[r] & D \arrow[r]           & G \arrow[lu] \arrow[llllld] &                               \\
	I \arrow[u] &             &             &             &                       &                             &                              
 \end{tikzcd} 
\end{center}

We see that this graph has three loops of length $3,7$ and $7$, hence  $A_R(\lambda)=\frac{1}{\lambda^3}+\frac{2}{\lambda^7}$. The characteristic polynomial is 
$P_{\delta,0}(\lambda)=\lambda^7-\lambda^4-2$, whose unique positive root is $\delta_0\approx 1.23175$, hence $h_F(b)=\ln(\delta_0)\approx 0.20844$.

\textbf{Case 3:} For $21/4<b<11/2,$ then $(x_0,-1)\in \overline{X_{3}X_{4}}$. This case corresponds with $b\in U_0=(r_0,q_0)$. For these values, $F_1$ covers $H$ and a part of $I$ while $G$ covers $K$ and the other part of $I:$

\begin{center} 
\begin{tikzcd}
	&             &             &             & K \arrow[d] \arrow[r] & C \arrow[r]                         & F_1 \arrow[lllld, bend right] \arrow[lllllldd, dashed, bend left] \\
	B \arrow[r] & E \arrow[r] & H \arrow[r] & A \arrow[r] & D \arrow[r]           & G \arrow[lu] \arrow[llllld, dashed] &                                                                   \\
	I \arrow[u] &             &             &             &                       &                                     &                                                                  
 \end{tikzcd} 
 \end{center}

The graph has two loops of length $3,7$ while the extended graph has four loops of length $3,7,7$ and $10.$ Following similar arguments than the ones in the previous cases we get that the characteristic polynomial of the incidence matrix of the first graph is $P_{\alpha,n}$ and the one associated with the extended graph is $P_{\gamma,n}(\lambda)=\lambda^{10}-\lambda^{7}-2 \lambda^{3}-1$, whose unique positive root is 
$\gamma_0\approx 1.25898$. Hence:
$$
0.14717\approx \ln(\alpha_0)\leq h_F(b)\leq \ln(\gamma_0)\approx 0.23031.
$$

\textbf{Case 4:} For $11/2\leq b\leq 20/3,$ then $(x_0,-1)\in \overline{X_{2}X_{3}}$. Now $F_1$ exactly covers $I\cup H$ and $G$ exactly covers $K:$


\begin{center} 
\begin{tikzcd}
	&             &             &             & K \arrow[d] \arrow[r] & C \arrow[r]  & F_1 \arrow[lllld, bend right] \arrow[lllllldd, bend left] \\
	B \arrow[r] & E \arrow[r] & H \arrow[r] & A \arrow[r] & D \arrow[r]           & G \arrow[lu] &                                                           \\
	I \arrow[u] &             &             &             &                       &              &                                                          
 \end{tikzcd} 
 \end{center}

In this case the graph is of Markov type and it has three loops of length $3,7$, and $10.$ The characteristic polynomial is $P_{\varphi,0}(\lambda)=\lambda^{10}-\lambda^{7}-\lambda^{3}-1$. The unique positive root of this polynomial is $\varphi_0\approx 1.20443$, hence $h_F(b)=\ln(\varphi_0)\approx 0.18600$.

\textbf{Case 5:} When $20/3<b<8$ we get that  $(x_0,-1)\in \overline{X_{2}P_{4}}.$ Following its orbit, we obtain:
$$
\begin{array}{l}
F_{x_0}:=F(x_0,-1)=(10-b,2b-11)\in \overline{Y_{1}Y_2},\\
F_{x_0}^2:=F^2(x_0,-1)=(20-3b,21-2b)\in \overline{P_{1}P_{12}},\\
F_{x_0}^3:=F^3(x_0,-1)=(5b-42,-1)\in \overline{P_{3}P_4}.
\end{array}$$
From this observation, we get that in order to understand the dynamics we must keep track of the third iterate of $(x_0,-1)$. Calling $(x_1,-1):=F^3(x_0,-1)$ we see that we have again five possibilities. For the first four we can add more points in the graph and the entropy can be studied as before.
Concerning the fifth one we have to consider its third iterate, $(x_2,-1),$ which belongs to $\overline{P_{3}P_4}$ and so on. This is an infinite process, and we will explore it in the next section obtaining, also, the proof of Proposition \ref{p:teo1}.

\subsection{Dynamics of $F^3|_{\Gamma\cap\{x=-1\}}$. Proof of Proposition \ref{p:teo1}}

As we have seen in the preliminary observations of the previous section, to keep on the study the dynamics and the entropy of $F|_{\Gamma}$, we need to know where are located the iterates of the point $(x_0,-1)$, with $x_0=b-10$, under $F^3$.

By using the notation introduced in Section \ref{s:s2}, we have that
$$
F^3(x,-1)=F_2\circ F_1\circ F_3(x,-1)=(4x-2+b,-1).
$$
So, the set $\Gamma\cap\{x=-1\}$ is invariant by $F^3$, and to study the iterates $(x_n,-1)=F^{3n}(x_0,-1)$ we can consider the linear recurrence   
$$\left\{
\begin{array}{l}
x_{n+1}=4x_n+b-2,\\
x_0=b-10,
\end{array}
\right.
$$ whose solution is:
$$x_n=\frac{(b-8)4^{n+1}+2-b}{3}.$$

As it follows from the considerations in the previous section,  the changes on dynamics, reflected on the changes on the directed graphs and their entropies, occurs in those bifurcation values of $b,$ such that $(x_n,-1)$ is either $X_2$ or $X_3$ or $X_4$ or $X_6.$ Then, let $p_n,q_n,r_n,s_n$ be the values of $b$ such that
\begin{align*}
&(x_n,-1)=X_2=\left(-\frac{b}{2},-1\right),\
&&(x_n,-1)=X_3=(1-b,-1),\\
&(x_n,-1)=X_4=\left(\frac{1-2b}{2},-1\right),
&&(x_n,-1)=X_6=\left(\frac{2-5b}{4},-1\right),
\end{align*}
respectively. A computation gives:
$$p_n=\frac{4(4\cdot 4^{n+1}-1)}{2\cdot 4^{n+1}+1}\,,\,\,\,q_n=\frac{8\cdot 4^{n+1}+1}{4^{n+1}+2}\,,\,\,\,r_n=\frac{16\cdot 4^{n+1}-1}{2\cdot 4^{n+1}+4}\,,\,\,\,s_n=\frac{2(4\cdot 4^{n+2}-1)}{4^{n+2}+11},$$ which are the recurrences defined in \eqref{e:rec}. As mentioned before, for any $n\in\N\cup\{0\}$:
$p_{n-1}<s_n<r_n<q_n<p_n$, and therefore we can define the intervals 
$S_n=(p_{n-1},s_n)$, $T_n=[s_n,r_n]$, $U_n=(r_n,q_n)$  and $V_n=[q_n,p_n]$ that appear in the statement of Proposition \ref{p:teo1}, which cover the whole interval $(4,8)$.


\begin{proof}[Proof of Proposition \ref{p:teo1}]
	The result for $n=0$  has already been obtained in the study of the first four cases in the Section \ref{ss:prelimgamma}. For $n\geq 1$ we introduce new points in the partition. \emph{We say that  $b$ is in the level $n,$ when $b\in S_n\cup T_n\cup U_n\cup V_n.$} In this case, the points $(x_0,-1),(x_1,-1),\ldots (x_{n-1},-1)$ belong to $\overline{{X_2}P_{4}}$ while $(x_n,-1)$ has the four different possiblities, that we will explore.
	Denoting by $F_{x_k}=F(x_k,-1)$ and $F^2_{x_k}=F^2(x_k,-1),$
	the partition now must incorporate the following intervals:
	
$$\begin{aligned}
		&F_1:=\overline{P_{3}(x_n,-1)},\,F_2:=\overline{(x_n,-1)(x_{n-1},-1)},\ldots, F_{n+1}:=\overline{(x_1,-1)(x_{0},-1)};\\
		&G:=\overline{(x_0,-1)P_4};\\
		&J_2:=\overline{Y_2F_{x_{n-1}}},\,J_3:=\overline{F_{x_{n-1}}F_{x_{n-2}}},\ldots, J_{n+1}:=\overline{F_{x_{1}}F_{x_{0}}};\\
		&C_1:=\overline{P_1F^2_{x_{n-1}}},\,C_2:=\overline{F^2_{x_{n-1}}F^2_{x_{n-2}}},\ldots,C_n:=\overline{F^2_{x_{1}}F^2_{x_{0}}},\,C_{n+1}:=\overline{F^2_{x_{0}}P_{12}}.
			\end{aligned}$$
We notice that the changes in the dynamics, accordingly with the values of $b$, are given by the changes in the coverings of $F_1,F_2,$ which are the two intervals that have the point $(x_n,-1)$ in its boundary.	Using this notation, we have also the following key observation:

\medskip

\textbf{Observation:} \emph{for $b$ in the level $n$, in the directed graphs, between $K$ and $C_1$ always there appear the $n$ groups $C_i\rightarrow F_i\rightarrow J_i$ for $i=2,3,\ldots ,n+1.$  }

\medskip

The statements (a)--(d) follow from the analysis of the following different possible cases:	
	
\textbf{Case 1:}  When $b\in S_n$, then $(x_n,-1)\in \overline{P_{3}X_{6}}$ which implies $F_{x_n}\in \overline{P_{6}Y_{6}}.$ Hence the image of $F_1$  is a part of $H$ and $F_2$ covers $I,J_2$ and the other part of $H.$ When $b\in S_1$, the directed graph is:

	{\scriptsize	\begin{center} \begin{tikzcd}
			&             &             &             & K \arrow[d] \arrow[r] & C_2 \arrow[r] & F_2 \arrow[r] \arrow[lllld, dashed, bend right] \arrow[lllllldd, bend left] & J_{2} \arrow[r] & C_1 \arrow[r] & F_1 \arrow[llllllld, dashed, bend left] \\
			B \arrow[r] & E \arrow[r] & H \arrow[r] & A \arrow[r] & D \arrow[r]           & G \arrow[lu]  &                                                                             &                 &               &                                         \\
			I \arrow[u] &             &             &             &                       &               &                                                                             &                 &               &                                        
		 \end{tikzcd} \end{center}}

		Looking at the graph corresponding to $b\in S_1,$ 
	and from the above observation, which indicates that we must introduce the $n$ groups $C_i\rightarrow F_i\rightarrow J_i$ for $i=2,3,\ldots ,n+1$, 	
		we see that when $b\in S_n$ then the graph has two loops of length $3$ and $7+3n$ while the extended graph has these two loops and two more loops (in dashed) of length $4+3n$ and $7+3n.$ By using the techniques applied in the first case in Section \ref{ss:prelimgamma}, we obtain that the characteristic polynomials of these two graphs are, respectively, $P_{\alpha,n}$ and $P_{\beta,n}$. From Descartes rule of signs, they both have only one positive root, $\alpha_n$ and $\beta_n$ respectively. Hence
$\ln(\alpha_n)\leq h_F(b)\leq \ln(\beta_n).$

\textbf{Case 2:}  When   $b\in T_n$, then $(x_n,-1)\in \overline{X_{4}X_{6}}$ which gives  $F_{x_n}\in \overline{Y_{4}Y_{6}}.$ Hence $F_1$ covers $H$ and $F_2$ covers $I,J_2.$
	When $b\in T_1$ the graph is Markov-type:

	{\scriptsize	\begin{center} \begin{tikzcd}
		&             &             &             & K \arrow[d] \arrow[r] & C_2 \arrow[r] & F_2 \arrow[r] \arrow[lllllldd, bend left] & J_{2} \arrow[r] & C_1 \arrow[r] & F_1 \arrow[llllllld, bend left] \\
		B \arrow[r] & E \arrow[r] & H \arrow[r] & A \arrow[r] & D \arrow[r]           & G \arrow[lu]  &                                           &                 &               &                                 \\
		I \arrow[u] &             &             &             &                       &               &                                           &                 &               &                                
	 \end{tikzcd} \end{center}}	
		
		For $b\in T_n$, the graph is also of Markov type. From the previous graph and the above observation, it has $2$ loops of length $7+3n$ and one of length $3.$ The characteristic polynomial is $P_{\delta,n}$ which has a unique positive root $\delta_n$, so $h_F(b)=\ln(\delta_n)$.
		
\textbf{Case 3:}  When   $b\in U_n$, then $(x_n,-1)\in \overline{X_{3}X_{4}}$ so $F_{x_n}\in \overline{P_{7}Y_{4}}.$ Hence the image of $F_1$  is $H$ and a part of $I$ and the image of $F_2$ is $J_2$ and the other part of $I.$ When $b\in U_1,$

{\scriptsize \begin{center} \begin{tikzcd}
		&             &             &             & K \arrow[d] \arrow[r] & C_2 \arrow[r] & F_2 \arrow[r] \arrow[lllllldd, dashed, bend left] & J_{2} \arrow[r] & C_1 \arrow[r] & F_1 \arrow[llllllld, bend left] \arrow[llllllllldd, dashed, bend left] \\
		B \arrow[r] & E \arrow[r] & H \arrow[r] & A \arrow[r] & D \arrow[r]           & G \arrow[lu]  &                                                   &                 &               &                                                                        \\
		I \arrow[u] &             &             &             &                       &               &                                                   &                 &               &                                                                       
	 \end{tikzcd} \end{center}}
	For $b\in U_n$, and introducing again the $n$ groups of length $3$ of the observation, we get that the graph has two loops of length $3$ and $7+3n$ while the extended graph has these two loops and two more loops (in dashed) of length $7+3n$ and $10+3n.$ The characteristic polynomials of these two graphs are $ P_{\alpha,n} $ and $P_{\gamma,n}$, respectively. By Descartes' rule of signs, each has exactly one positive root, denoted by $\alpha_n$ and $\gamma_n$, respectively. Therefore, it follows that  
$\ln(\alpha_n) \leq h_F(b) \leq \ln(\gamma_n)$.	
		
\textbf{Case 4:}  When   $b\in V_n$, then $(x_n,-1)\in \overline{X_{2}X_{3}}$ hence  $F_{x_n}\in \overline{Y_{2}P_{7}}.$ Therefore $F_1$ covers $H,I$ and $F_2$ covers $J_2.$
When $b\in V_1,$

	{\scriptsize \begin{center} \begin{tikzcd}
			&             &             &             & K \arrow[d] \arrow[r] & C_2 \arrow[r] & F_2 \arrow[r] & J_{2} \arrow[r] & C_1 \arrow[r] & F_1 \arrow[llllllld, bend left] \arrow[llllllllldd, bend left] \\
			B \arrow[r] & E \arrow[r] & H \arrow[r] & A \arrow[r] & D \arrow[r]           & G \arrow[lu]  &               &                 &               &                                                                \\
			I \arrow[u] &             &             &             &                       &               &               &                 &               &                                                               
		 \end{tikzcd} \end{center}}

Hence, for $b\in V_n$, the graph is of Markov type and it has $3$ loops of length $3,7+3n$ and $10+3n.$ The characteristic polynomial is $P_{\varphi,n}$. Again, it has a unique positive root $\varphi_n$ so $\ln(\varphi_n)$. This ends the proof of statements (a)--(d).

To prove Eq. \eqref{e:relpositions}, we first notice that since the polynomials $P_{\alpha,n}(\lambda)$,
$P_{\beta,n}(\lambda)$, $P_{\delta,n}(\lambda)$, $P_{\gamma,n}(\lambda)$, $P_{\varphi,n}(\lambda)$ take negative values at $\lambda=1$ and have a positive coefficient in their leading term, so their unique positive roots $\alpha_n,\beta_n,\delta_n,\gamma_n,$ and $\varphi_n$
are located in $\lambda>1$. Now, to study their relative positions we easily see that they correspond with the abscissa of the intersection of the graphs of the functions $f_n(\lambda)=\lambda^{3n}$ and
the functions
\begin{align*}
&g_{\alpha}(\lambda)=\frac{1}{\lambda^4(\lambda^3-1)},&&
g_{\beta}(\lambda)=\frac{\lambda^3+2}{\lambda^4(\lambda^3-1)},&&
g_{\delta}(\lambda)=\frac{2}{\lambda^4(\lambda^3-1)},\\
&g_{\gamma}(\lambda)=\frac{2\lambda^3+1}{\lambda^7(\lambda^3-1)},&& 
g_{\varphi}(\lambda)=\frac{\lambda^3+1}{\lambda^7(\lambda^3-1)}.&&
\end{align*}
respectively. Set $\lambda>1$.  We  have
$$
g_{\alpha}(\lambda)=\frac{1}{\lambda^4(\lambda^3-1)}=\frac{\lambda^3}{\lambda^7(\lambda^3-1)}<\frac{\lambda^3+1}{\lambda^7(\lambda^3-1)}=g_{\varphi}(\lambda);
$$

$$
g_{\varphi}(\lambda)=\frac{1}{\lambda^4(\lambda^3-1)}+\frac{1}{\lambda^7(\lambda^3-1)}<\frac{1}{\lambda^4(\lambda^3-1)}+\frac{1}{\lambda^4(\lambda^3-1)}=g_{\delta}(\lambda),
$$

$$
g_{\delta}(\lambda)=\frac{2}{\lambda^4(\lambda^3-1)}<\frac{2\lambda^3}{\lambda^7(\lambda^3-1)}<\frac{2\lambda^3+1}{\lambda^7(\lambda^3-1)}=g_{\gamma}(\lambda),
$$
and
$$
g_{\gamma}(\lambda)=\frac{2\lambda^3+1}{\lambda^7(\lambda^3-1)}<\frac{2\lambda^3+\lambda^6}{\lambda^7(\lambda^3-1)}=\frac{2+\lambda^3}{\lambda^4(\lambda^3-1)}=g_{\beta}(\lambda).
$$
Hence 
$$g_\alpha(\lambda)<g_\varphi(\lambda)
<g_\delta(\lambda)<g_\gamma(\lambda)
<g_\beta(\lambda)\mbox{ for } \lambda>1$$
(see Figure \ref{f:fig3})
and, therefore, inequalities \eqref{e:relpositions} are proved.

\begin{figure}[ht]
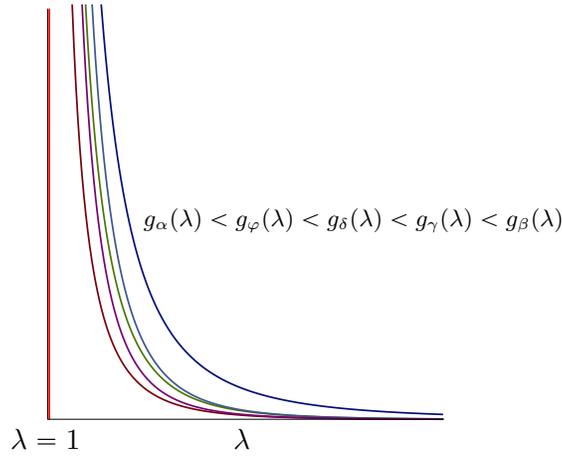

	 	\centering
	 	
	 	\begin{lpic}[l(2mm),r(2mm),t(2mm),b(2mm)]{gr4b(0.40)}

\lbl[c]{100,-5; $\lambda$}
\lbl[c]{10,-5; $\lambda=1$}

\lbl[l]{55,95; {\footnotesize $g_\alpha(\lambda)<g_\varphi(\lambda)<
g_\delta(\lambda)<g_\gamma(\lambda)<
g_\beta(\lambda)$}}
\end{lpic}
\caption{Graphs of the functions $g_\alpha(\lambda)$ (brown), 	$g_\varphi(\lambda)$ (purple),
$g_\delta(\lambda)$ (green), $g_\gamma(\lambda)$ (light blue), and
$g_\beta(\lambda)$ (blue)
for  $\lambda>1$.}\label{f:fig3}
\end{figure}

By using  inequalities \eqref{e:relpositions} we prove that
all the roots  $\alpha_n$, $\delta_n$, $\gamma_n$ and $\varphi_n$ tend to $1$ when $n$ tends to infinity, by proving it 
only for the sequence of roots $\beta_n$. To do this, we simply study the relative positions of the graphs of the functions $f_n(\lambda)=\lambda^{3n}$ and $g_\beta(\lambda)$ for $\lambda>1$ (see Figure \ref{f:fig4}).

\begin{figure}[ht]
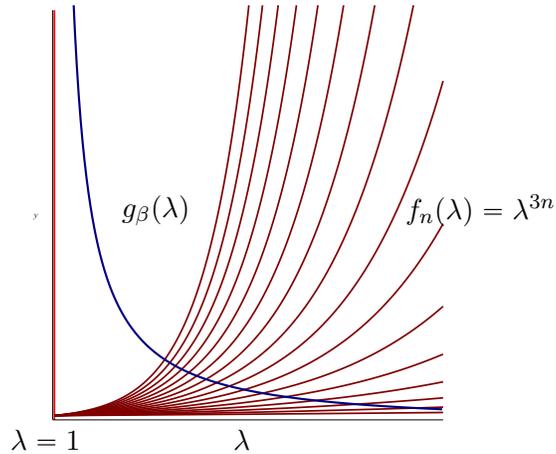

	 	\centering
	 	
	 	\begin{lpic}[l(2mm),r(2mm),t(2mm),b(2mm)]{gr5b(0.40)}

\lbl[c]{100,-5; $\lambda$}
\lbl[c]{10,-5; $\lambda=1$}

\lbl[l]{45,100; $g_\beta(\lambda)$}
\lbl[l]{175,100; $f_n(\lambda)=\lambda^{3n}$}
\end{lpic}
\caption{Graphs of $g_\beta(\lambda)$ (blue) and the  family of functions  $f_n(\lambda)=\lambda^{3n}$ (brown)
for  $\lambda>1$. }\label{f:fig4}
\end{figure}

Indeed, for all $\varepsilon>0$ and for all
$$
n>\frac{\ln(g_\beta(1+\varepsilon))}{3\ln(1+\varepsilon)},
$$
we get that $f_n(\lambda)=(1+\varepsilon)^{3n}>g_\beta(1+\varepsilon)$ and, therefore,  $\beta_n<1+\varepsilon$. Hence $\lim\limits_{n\to\infty}\beta_n=1$, which concludes the proof of the result.~\end{proof}

\section{Approximations of $\alpha$ and $\beta$ and proof of Proposition~\ref{pro:ab}}\label{s:alphaibeta}

Theorem \ref{t:teoD} states that, when $a=-1$, there exist
$\alpha\in (-112/137,-13/16)\approx(-0.8175,$ $-0.8125),$  and $\beta\in(603/874,563/816)\approx(0.6899,0.6900)$
such that the entropy $h_F(b)$ transitions from zero to positive for $b = \alpha$ and $b = \beta$, respectively.  
In this section, we provide a constructive method to obtain rational approximations of $\alpha$ and $\beta$ and we prove Proposition~\ref{pro:ab}.

\subsection{Rational approximations of $\alpha$}\label{ss:alpha}

The invariant graph $\Gamma$ for $a=-1$ and $-1< b\leq -3/4$ is given in Figure~\ref{particio-1-3/4}.

We use the following notation: $P_{1}=(0, b +1)$, $P_2=(-b -2, -1)$,  $P_3=(b +2, -3)$, $P_4=(b +4, -1+2 b)$,  $P_5=(-b +4, 4 b +3)$, $P_{6}=(-5b, 4b+7)$. For $i=1,\ldots,6,$ $P_i=F^{i-1}((0,b+1).$ 
Also $R_1=(-b-1,0)$,  $R_2=(b,-1)$,  $R_{3}=(-b,-1+2 b)$, $R_{4}=(-3 b, -1+2 b)$,  $R_{5}=(-5 b, -1)$,  $R_{6}=(-5 b, -4 b -1)$, $R_7=(-b,1).$ For $i=1,\ldots,7,$ $R_i=F^{i-1}((-b-1,0)).$ Moreover
$Q=(0,b-1),$ $T_1=(-5b,0)$ and $T_2=F(T_1)=(-5b-1,-4b).$ We have $F(R_7)=F(T_2)=P_2$ and $F(S)=R_3.$

In \cite[Proposition 27(d)]{CGMM} we prove that when $a=-1$ and $b\in [-112/137,-13/16]$ the subinterval $\Pi=\overline{R_7P_7}=U\cup V\cup A \subset \Gamma$ is invariant by $F^6.$ Here $P_7:=F(P_6)=(-9b-8,-8b-7)\in \overline{R_1P_2}$ and $U=\overline{R_7P_1}\,,\,V=\overline{P_1R_1}$ and $A=\overline{R_1P_7}.$ This interval $\Pi$ is visited for all elements of $\Gamma$ except for the points of a repulsive $6$-periodic orbit.

\begin{figure}[H]
	\footnotesize
	\centering
	\begin{lpic}[l(2mm),r(2mm),t(2mm),b(2mm)]{novafig7(0.53)}
		\lbl[r]{47,92; $R_1$}
		\lbl[r]{60,85; ${\color{blue} \mathbf{V}}$}
		\lbl[r]{50,83; ${\color{blue} \mathbf{A}}$}
		\lbl[c]{31,80; $P_7$}	

		\lbl[r]{52,98; $P_1$}
		\lbl[r]{22,67; $P_2$}
		\lbl[l]{36,67; $R_2$}
		\lbl[r]{50,42; ${\color{blue} \mathbf{Q}}$}

		\lbl[r]{72,22; $R_3$}
		\lbl[c]{79,12; $P_3$}

		\lbl[l]{65,100;${\color{blue} \mathbf{U}}$}

		\lbl[c]{116,20; $R_4$}
		\lbl[l]{132,20; $P_4$}
		\lbl[l]{160,65; $R_5$}
		\lbl[l]{158,94; $T_1$}
		\lbl[c]{177,86; $P_5$}
		\lbl[l]{158,149; $R_6$}
		\lbl[c]{154,185; $P_6$}
		\lbl[c]{131,174; $T_2$}
		\lbl[c]{68,117; $R_7$}

	\end{lpic}
	\caption{ The graph $\Gamma$ for $a=-1$ and $-1< b\leq -3/4$. When $-112/137\leq b\leq -13/16$ the point $P_7\in \overline{R_1P_2}$.}\label{particio-1-3/4}
\end{figure}

In the proof of this Proposition, see \cite{CGMM}, we explicitly compute the map $F^6$ restricted to the interval $\Pi.$  Indeed, the points of $\Pi$ write as $(x,x+b+1)$ where $x\in [-9b-8,-b],$  and  $F^6(x,x+b+1)=(g_1(x),g_1(x)+b+1)$ where $g_1(x)$ is certain explicit map. In~\cite{CGMM} it is proved   
that $g_1(x)$ is semiconjugated to a trapezoidal map of those studied in \cite{BMT}. This allows us to use the results of that reference to obtain a transition from zero to positive entropy. Indeed, some computations show that 
$g_1(x)$ is semiconjugated to the piecewise continuous linear map of the interval 
$g_2(x):[0,\frac{9b+8}{8(b+1)}]\longrightarrow [0,\frac{9b+8}{8(b+1)}]$ defined by $$g_2(x)=\begin{cases} 16x-\frac{16b+13}{b+1} & \mbox{if $x\in [0,u_1]$},\\
	\frac{9b+8}{8(b+1)} & \mbox{if $x\in [u_1,u_2]$},\\
	-8x+\frac{9b+8}{b+1} & \mbox{if $x\in [u_2,\frac{9b+8}{8(b+1)}]$}.
\end{cases}
$$
where $u_1=\frac{137b+112}{128(b+1)}$ and $u_2=\frac{7(9b+8)}{64(b+1)}.$  
Recall, that as we have already explained in Section~\ref{a:prelim}, the semiconjugation between $g_1(x)$ and $g_2(x)$ does not affect the entropy calculation. Intuitively this fact can be understood because  it simply corresponds to remove an interval of constancy of the map.

Now we extend this map on a bigger interval $[x_1,x_2]\supset \left[0,\frac{9b+8}{8(b+1)}\right]$ to get a trapezoidal map $g_3(x):$

$$g_3(x)=\begin{cases} 16x-\frac{(16b+13)}{b+1} & \mbox{if $x\in [x_1,u_1]$},\\
	\frac{9b+8}{8(b+1)} & \mbox{if $x\in [u_1,u_2]$},\\
	-8x+\frac{9b+8}{b+1} & \mbox{if $x\in [u_2,x_2]$}.
	\end{cases}
$$
Here $x_1=\frac{16b+13}{15(b+1)}<0$ is the repulsive fixed point of $L(x):=16x-\frac{(16b+13)}{b+1}$ and $x_2=\frac{119b+107}{120(b+1)}>\frac{9b+8}{8(b+1)}$ satisfies that $g_3(x_2)=x_1.$  See Figure~\ref{GrafT0} for an illustration of this. We note that since $x_1$ is repulsive for each $x\in(x_1,x_2)$ there exists  $n$ such that $g_3^n(x)\in \left(0,\frac{9b+8}{8(b+1)}\right).$ So the dynamics of $g_2$ can be studied analyzing $g_3$ and vice versa.  We omit all the details because the reader can find them, fully developed, in \cite{CGMM}.

\begin{figure}[ht]
	\centering
	\begin{lpic}[l(2mm),r(2mm),t(2mm),b(2mm)]{grtb(0.40)}
		\lbl[r]{3,10; {$x_1$}}
		\lbl[r]{3,190; {$x_2$}}
		\lbl[c]{4,-5; {$x_1$}}
		\lbl[c]{196,-5; {$x_2$}}
		\lbl[c]{70,-5; {$0$}}
		\lbl[c]{90,-5; {$u_1$}}
		\lbl[c]{130,-5; {$u_2$}}
		\lbl[c]{166,-5; {$\frac{9b+8}{8(b+1)}$}}
	\end{lpic}
	\caption{Sketch of the graphic of $g_3(x)$ in blue and the graphic of $g_2(x)$
		inside the red box. The graphic is not to scale, \cite{CGMM}.}
	\label{GrafT0}
\end{figure}

 After a rescaling of the interval $[x_1,x_2]$ to the interval $[0,1]$, we obtain that the map $g_3(x)$ is conjugated with the trapezoidal map $g_4=T_{1/16,1/8,Z}$ with $Z=\frac{55 b}{16(3b-1)}$ according with the notation in \cite{BMT}. Hence, by using the results in this reference, we  know \emph{that 
there exists $\alpha\in (-112/137,-13/16)$ such that $h_F(b)=0$  for $b\in[-112/137,\alpha]$ while $h_F(b)>0$ and non-decreasing when $b\in(\alpha,-13/16]$.}

To determine sharp bounds of $\alpha$ it suffices to fix our attention to the map $g_2.$  To simplify the calculations that follow, instead of working with the  trapezoidal map $T_{1/16,1/8,Z}$, we will use the following one, $\varphi:[0,1]\longrightarrow[0,1]$, which is also conjugate to the map $g_2(x).$
$$\varphi(x)=\begin{cases} 16x+d & \mbox{if $x\in [0,\frac{1-d}{16}]$},\\
	1 & \mbox{if $x\in [\frac{1-d}{16},\frac7 8]$},\\
	-8x+8 & \mbox{if $x\in [\frac78,1]$}.
\end{cases}
$$
where $d=-\frac{8(16b+13)}{9b+8}.$  The procedure to approach $\alpha$ will consist on providing upper and lower values of this value. The idea is very simple because we already know that in this range of values of~$b$ the entropy $h_F(b)$ is non-decreasing:

\begin{itemize}
	\item To get upper bounds it suffices to find a value of $b,$ say $b_p$, such that $\varphi$ has a period orbit in $[0,1]$ of minimal period $p=m 2^N$ for some $0<N\in\N,$ and $m$ odd (in fact we will always take $m=3$).  
By the Bowen-Franks' Theorem, \cite[Theorem 1]{BF}, the entropy satisfies:
$
h_{\varphi}(b_p)> \ln(2)/p>0.
$
So, using Lemma \ref{entroo} in Section~\ref{a:prelim}, $h_F(b_p)= h_{g_2}(b_p)/6=h_\varphi(b_p)/6>0$ (recall that $g_2$ is, essentially, $F^6|_{\Pi}$) and therefore $\alpha<b_p$. 
Note that, a priori, this value of $b_p$ exists because the maps $\varphi$ are also conjugate to the trapezoidal map family and thus encompass all possible dynamics of a unimodal map, \cite{BMT}.

The fact that $h_\varphi(b_p)>0$ can also be seen by studying the Markov partition induced by this periodic orbit.

	\item  To get lower bounds it suffices to find a value of $b$, say $b_p$  such that $\varphi$ has a period orbit in $[0,1]$ of minimal period $p=2^N,$ for some $0<N\in\N,$ and such that the Markov partition induced by this periodic orbit gives zero entropy. Then  $b_p<\alpha.$
\end{itemize}

Let us start with the upper bound $b_3.$ We impose that the orbit starting at $x=1$ is $3$-periodic for $\varphi.$ We have  $1\to 0 \to d \to \varphi(d).$  The condition $\varphi(d)=1$ is equivalent to 
\[
\frac{1-d}{16}\le d \le \frac 78 \quad\Longleftrightarrow\quad  \frac 1{17}\le d \le  \frac78  \quad     \Longleftrightarrow\quad  -\frac{888}{1087}\le b\le  -\frac{1776}{2185}.
\] 
where we have used that $b=-\frac{8(d+13)}{9d+128}.$  So, we know that by taking $b_3=-\frac{888}{1087}\approx -0.8169,$  we obtain that $h_{\varphi}(b_3)>\ln(2)/3>0$ and, therefore, $\alpha<b_3.$

Similarly, by taking $d=\frac{7295}{8191},$ we obtain that its orbit by $\varphi$ is $6$-periodic
\[
1\to 0\to \frac{7295}{8191} \to \frac{7168}{8191}\to\frac{8184}{8191}\to\frac{56}{8191}\to1.
\]
Hence by taking $d=\frac{7295}{8191},$ we obtain that $b_6=-\frac{910224}{1114103}\approx -0.81700166$,
$h_{\varphi}(b_6)>\ln(2)/6>0$ and $\alpha<b_6.$

By following this approach we have proved that by taking
\begin{align}\label{fitasup}
b=b_{24}&=-\frac{1049417824596806956103568}{1284474531463219438945271}\nonumber\\&\approx -0.817001660127394075579379106922368833240
\end{align}	
 the map $\varphi$ has a $24$-periodic orbit, $h_\varphi(b_{24})>\ln(2)/24>0.$ So $\alpha<b_{24}.$

To get lower bounds for $\alpha$ we will search for $2^N$-periodic orbits of $\varphi$ for $N=2,3,4,5$ with a given itinerary. The fact that  these itineraries  provide periodic orbits of these periods that give rise to Markov partitions with zero entropy is a well known and established fact in the study of unimodal maps, see~\cite{Col}. 

To codify the orbits, as usual we will call the intervals $L:=\left[0,\frac{1-d}{16}\right],$ $C:=\left[\frac{1-d}{16},\frac{7}{8}\right],$ 
and $R:=\left[\frac{7}{8},1\right].$ Then, if we consider  $x_0\in [0,1],$  $x_{n+1}=\varphi(x_n),$ and the sequence $x_0,x_1,x_2,x_3,$ $\ldots, x_n,\ldots,x_m$ we will say that it has for instance the itinerary $RLLR\cdots L \cdots C$ if $x_0,x_3\in R,$  $x_1,x_2,x_n\in L$ and $x_m\in C.$ Similarly, we introduce the linear maps $L(x)=\varphi(x)|_{L}=16x+d$ and $R(x)=\varphi(x)|_{R}=-8x+8.$

 With this notation,  the itinerary of $2^N$-periodic orbits that provide zero entropy can be obtained from a seed itinerary and the so called ``$*$-product''. In a few words, given the itinerary $\mathcal{R}=RRRR\ldots$ and any itinerary $\mathcal{S}=S_1S_2\ldots S_k,$ where $S_j\in\{L,R\}$ we define a new itinerary of length $2k,$ where $T_j\in\{RL, RC, RR\}$ as
 \[
 \mathcal{R}*\mathcal{S}=T_1T_2\ldots T_{2k},\quad\mbox{where}\quad T_j=\begin{cases}
 	RL &\mbox{if}\quad S_j=R,\\
 	RC &\mbox{if}\quad S_j=C,\\
 	RR &\mbox{if}\quad S_j=L.
 \end{cases}
 \]
Hence, if we start with the seed $\mathcal{S}_2=RC$ we obtain that  $\mathcal{S}_4=\mathcal{R}*\mathcal{S}_2=RLRC.$ Similarly,
\begin{align*}
\mathcal{S}_8&=\mathcal{R}*\mathcal{S}_4=RLRRRLRC,\\	
\mathcal{S}_{16}&=\mathcal{R}*\mathcal{S}_8=RLRRRLRLRLRRRLRC,\\
\mathcal{S}_{32}&=\mathcal{R}*\mathcal{S}_{16}=RLRRRLRLRLRRRLRRRLRRRLRLRLRRRLRC.
\end{align*}  
Following~\cite{Col}, to obtain periodic orbits that give a Markov partition with zero entropy, we will search for $2^N$-periodic orbits with itineraries $\mathcal{S}_{2^N}$ for $N=2,3,4$ and $5.$

To find a $4$-periodic orbit we impose that $\varphi^4(1)=1,$ with itinerary $RLRC.$ This is equivalent to impose that $R(L(R(1)))\in C.$  In other words, $1\stackrel R \to 0\stackrel L\to d\stackrel R \to 8-8d$ and $8-8d\in C,$ that its
\[
\frac{1-d}{16}\le 8-8d \le \frac 78 \quad\Longleftrightarrow\quad  \frac{57}{64}\le d \le  1   \quad\Longleftrightarrow\quad  -\frac{112}{137} \le b \le -\frac{7112}{8705},
\] 
where we have used again that $b=-\frac{8(d+13)}{9d+128}.$ Then by taking $b=b_4:=-\frac{7112}{8705}\approx -0.8170017$ we know that $h_\varphi(b_4)=0$ and $\alpha> b_4.$ 

For $8$-periodic orbits the desired itinerary is $\mathcal{S}_8=RLRRRLRC$ and the orbit with smaller $d$ is when  $d={\frac {933761}{1048449}},$ and then
\[
1\stackrel R\to 0\stackrel L\to {\frac {933761}{1048449}}\stackrel R\to{\frac {917504}{1048449}}\stackrel R\to{\frac {
		1047560}{1048449}}\stackrel R\to{\frac {7112}{1048449}}\stackrel L\to{\frac {1047553}{1048449}}\stackrel R\to
{\frac {7168}{1048449}}\stackrel C\to1.
\]
The value of $b$ corresponding to this $d$ is $b=b_8:=-\frac{116508784}{142605321}\approx -0.8170016601273945.$ Again $h_\varphi(b_8)=0$ and $\alpha>b_8.$

Finally, by taking
\[
b=b_{32}:=-\frac{140850476140085945702816746162288}{172399253286857828660669132569609}
\]
we get a $32$-periodic orbit of $\varphi$ starting at $1,$  $h_\varphi(b_{32})=0$ and $\alpha>b_{32}.$ Since $|b_{24}-b_{32}|<4\times 10^{-40},$ it follows that all the digits of the expression of $b_{24}$ given in~\eqref{fitasup} are also right digits of $\alpha.$ This proves the first part of Proposition~\ref{pro:ab}.

\subsection{Rational approximations of $\beta$}\label{ss:beta}

The invariant graph $\Gamma$ for $a=-1$ and $-2/3< b\leq 5/7$ is given in Figure~ \ref{f:A}. The points appearing there are  $P_{1}=(-b -2, -1)$, $P_{2}=(b +2, -3)$, $P_{3}=(b +4, 2 b -1
)$, $P_{4}=(-b +4, 5)$, $Q=(0, 7 b -5)$, $R_1=(0, 2 b -1)$, $R_2=(-2 b, 1-b)$, $R_3=(3 b -2, -1)$,   $R_4=(3 b -2, 4 b -3)$,  $R_5=(-b
, 8 b -5)$, $R_6=(-7 b +4, -8 b +5)$,  $R_7=(15b-10,-14b+9)$, $S=(0, b +1)$, $T_1=(-b, 0)$, $T_2=(b -1, 0)$, $W=(1-3 b, 0)$, 
$X_1=(0, -1)$,
$X_2=(0, b -1)$, $X_3=(-b, 2 b -1)$,
$X_4=(-b, 1-2 b)$, $X_5=(3 b -2
, 1-2 b)$, $X_6=(5 b -4, 2 b -1)$, $X_7=(-7 b +4, 4 b -3)$, $Y_1=(-7 b +4, 0)$, $Y_2=(7 b -5, -6 b +4)$,   $Z_1=(7 b -5, 0)$, $Z_{2}=(-7 b +4, 8 b -5)$ and
$Z_3=(-b, 9-14 b)$. 

Taking the point $R_1=(0, 2 b -1)$ and their iterates $R_{i+1}=F(R_i)$, we get the points $R_8=(29b-20,2b-1)$ and 
$R_{15}=(300-435b,2b-1)$. In \cite[Proposition 31(c)]{CGMM} we prove that for $b\in [603/874,563/816]$ the subinterval $\Sigma=\overline{R_{15}R_8}\subseteq \Gamma$  is invariant by $F^7$, and that this interval is visited by any point of $\Gamma$ except a 3-periodic orbit, 
a 7-periodic orbit, a 4-periodic orbit (all of them repulsive) and the preimages of these orbits.

\vfill
\newpage

\vspace{1.5cm}

\begin{figure}[H]
	\footnotesize
	\centering
	\begin{lpic}[l(2mm),r(2mm),t(2mm),b(2mm)]{am-cas-A-zoom-v3b(0.50)}
		\lbl[l]{3,185; $\boxed{a=-1,\,2/3< b\leq 5/7}$}
		
		\lbl[l]{95,173; $S$}
		\lbl[l]{28,111; $R_2$}
		\lbl[l]{48,93; $W$}
		\lbl[l]{58,82; $X_7$}
		\lbl[l]{73,79; $X_4$}
		\lbl[r]{60,70; $R_6$}
		\lbl[l]{73,60; $Z_3$}
		\lbl[l]{87,88; $Y_2$}
		\lbl[l]{96,78; $X_2$}
		\lbl[l]{108,80; $X_5$}
		\lbl[l]{104,48; $R_3$}
		\lbl[l]{93,48; $X_1$}
		\lbl[l]{108,88; $R_4$}
		\lbl[l]{103,93; $Q$}
		\lbl[c]{100,101; $Z_1$}
		\lbl[c]{80,119; $X_6$}
		\lbl[c]{66,119; $X_3$}
		\lbl[c]{70,128; $R_5$}
		\lbl[c]{62,128; $Z_2$}
		\lbl[l]{54,101; $Y_1$}
		\lbl[l]{103,119; $R_1$}
		\lbl[l]{71,101; $T_1$}
		\lbl[l]{86,101; $T_2$}
		\lbl[l]{126,62; $R_7$}

	\end{lpic}
	
	\caption{Detail of the  graph $\Gamma$ for $a=-1$ and $2/3< b\leq 5/7$  and,  beside, a larger view~\cite{CGMM}. }\label{f:A}

\vspace{-12,8cm}\hspace{9cm}\begin{lpic}[l(2mm),r(2mm),t(2mm),b(2mm),figframe(0.2mm)]{am-cas-A(0.22)}

	\lbl[l]{60,117; $S$}
	\lbl[l]{27,84; $R_2$}
	\lbl[r]{95,44; $R_3$}
	\lbl[r]{25,44; $P_{1}$}
	\lbl[l]{136,10; $P_{2}$}
	\lbl[l]{170,95; $P_{3}$}
	
	\lbl[l]{152,186; $P_{4}$}
\end{lpic}
\end{figure}

\vspace{7.9cm}

In the proof of the mentioned Proposition we show that  the map  $F^7\vert_\Sigma$ is given by $F^7(x,2b-1)=(k_1(x),2b-1)$ with $x\in[-435b+300,29b-20]$,
where
$$k_1(x)=\begin{cases} 16x+4-3b &\mbox{if }x\in [-435b+300,2b-3/2],\\29b-20 & \mbox{if }x\in [2b-3/2,0],\\ -16x+29b-20 & \mbox{if }x\in [0,29b-20].	
						\end{cases}	$$ 
As in the previous case, the map $k_1(x)$  can be extended to a   trapezoidal map $k_2=T_{1/16,1/16,Z}$ with $Z= \frac{45-60b}{48b-29}$ in the notation of \cite{BMT}, defined in a larger interval in such a way that both have essentially the same dynamics and moreover share  entropy. Then, by using the results of that paper,  we can  prove that 
\emph{there exists $\beta\in(603/874,563/816)$ such that
$h_F(b)=0$ for $b\in(603/874,\beta]$ while  
$h_F(b)>0$ and is non-decreasing when $b\in(\beta,563/816].$ }

To determine sharp approximations of $\beta$ we will use that $k_1(x)$ is   also conjugated to the piecewise continuous linear map of the interval $[0,1],$
$$ \psi(x)=\begin{cases} 16x+d & \mbox{if $x\in \left[0,\frac{1-d}{16}\right]$},\\
	1 & \mbox{if $x\in \left[\frac{1-d}{16},\frac{15}{16}\right]$},\\
	-16x +16& \mbox{if $x\in \left[\frac{15}{16},1\right]$},
\end{cases}
$$
which is very similar to the map $\varphi$ used to determine approximations of~$\alpha.$ By following the same steps that in the previous case we obtain that
for \[b=b_{24}:={\frac {945506314303393205598153}{1370433212950874384162254}}\] it holds that $h_{\psi}(b_{24})>\ln(2)/24>0.$

Similarly, for \[b=b_{32}:={\frac {798396920638883099973166531706985228123}{
	 		1157210312199077596904301690272087447914}}
	 \] we get that $h_{\psi}(b_{32})=0.$

Hence $\beta \in (b_{32},b_{24})$ and since $|b_{24}-b_{32}|<5\times 10^{-50},$ we have that all shown digits of
\[
b_{24}\approx 0.6899324282045742867004889129507817387052603507745,
\]
are also right digits of $\beta.$ This proves the second part of Proposition~\ref{pro:ab}.

\section{Proof of Theorem~\ref{t:teonou} }\label{s:full}

An interesting fact that appears when studying the maps $F$, is that all their dynamic complexity for $a<0$, as reflected in Theorems \ref{t:teoB} and \ref{t:teoD}, does not appear in numerical simulations, where only periodic orbits are observed as $\omega$-limit sets.

In Theorem \ref{t:teoC}, it is proved that for each value of $a<0$ and arbitrary $b$, there exists \emph{an open and dense set} in each of the invariant graphs such that, for points in this set, there are at most three distinct 
$\omega$-limits which, when $b/a\in \mathbb{Q},$ are attractive periodic orbits. We remark that this last condition is always satisfied in numerical experiments. 

However, to demonstrate that in simulations we will only observe periodic orbits, is not enough to have an open and dense set of initial conditions in each graph converging to a periodic orbit: it is necessary to prove that this set has full measure in the graph. As we explain in \cite{CGMM}, we believe that this is the case in all the graphs.   
This is what we prove in Propositions \ref{p:propofull} and \ref{p:x} of this work for some particular range of the parameters. To understand the nature of this full-measure set, we recall some issues we identified in the proof of Theorem \ref{t:teoC}: 
\begin{itemize}
\item For each invariant graph, there exist some concrete edges that collapse to a point under the action of $F$. The open and dense set in the statement of  theorem is the set of preimages of these edges that, inspired by the notation introduced in \cite{BMT}, we call  \emph{plateaus}. The $\omega$-limits of these plateaus are the ones referred in the statement. These plateaus are the edges of $\Gamma$ in $Q_1$ whose axes have slope $1$, and the edges of $\Gamma$ in $Q_3$ whose axes have slope $-1$. 
\item  The $\omega$-limits of the plateaus are the only visible ones, in long-term, by numerical simulation.  This is because on the rest of the edges, the dynamics is expansive (and consequently repulsive). For this reason, and in particular, the great bulk of the periodic orbits are repulsive \cite[Lemma 22(c)]{CGMM}, and are not visible in numerical simulations.
\end{itemize}

We split the proof of Theorem~\ref{t:teonou} into two cases because the methods used are completely different.

\subsection{Case $a=-1$ and $b<-2$}

In this range of parameters, by using again~\eqref{conj} we can take $a=-1$ and $b<-2$. For them the associated invariant graph $\Gamma$ is the topological circle given in Figure~\ref{ff:1}, and $F|_{\Gamma}$ is conjugated to a degree-1 circle map, so its dynamics can be described in terms of an associated rotation number, \cite{ALM}. \emph{It is important to notice that the graph has only one plateau: $\overline{R_2S}$.} Other points in the graph given in Figure~\ref{ff:1} are  $P_{1}=(-b -2, -1)$, $P_{2}=(-b -2, -3)$, $P_{3}=(-b, -5)$, $P_{4}=(-b +4, -5)$, $P_{5}=(-b +8, -1)$, $P_{6}=(-b +8, 7)$, $P_{7}=(-b, 1)$, $R_{1}=(-b +8, 0)$, $R_{2}=(-b +7, 8)$, and $S=(-b -1, 0)$. 

\begin{figure}[H]
	\footnotesize
	\centering
	\begin{lpic}[l(2mm),r(2mm),t(2mm),b(2mm)]{am-cas-1-especial(0.47)}
		\lbl[l]{27,180; $\boxed{a=-1,\,b\leq -2}$}
		\lbl[l]{52,83; $S$}
		\lbl[r]{45,65; $P_1$}
		\lbl[r]{45,40; $P_2$}
		\lbl[l]{70,10; $P_3$}
		\lbl[l]{121,10; $P_4$}
		\lbl[l]{177,67; $P_5$}
		\lbl[l]{177,83; $R_1$}
		\lbl[l]{177,169; $P_6$}
		\lbl[c]{162,185; $R_2$}
		\lbl[l]{65,93; $P_7$}
	\end{lpic}
	\caption{The graph $\Gamma$ for $a=-1$ and $b\leq -2$, \cite{CGMM}.}\label{ff:1}
\end{figure}

The Proposition 25 of \cite{CGMM} fully describes the dynamics of $F|_{\Gamma}$ for $a=-1$ and $b<-2$. We summarize it:
\emph{\begin{itemize}
\item The map $F$ has the fixed point $p=(-b,-1)\in Q_4$. Moreover,
the map $F\vert_{\Gamma}$ is conjugated to a degree-1 circle map with zero entropy. Its rotation number is $1/7.$ 
\item It has two periodic orbits: (a) The $7$-periodic orbit, $\mathcal{P}$, given by the points $P_{i+1}=F(P_i)$ with $i=1,\ldots,6$,  which is attractive, and which is the $\omega$-limit of the unique plateau of $\Gamma$, $\overline{SR_2}$. (b) The $7$-periodic orbit $\mathcal{Q}$, characterized by the initial condition  $\left(-b-\frac{16}{15},-\frac{1}{15}\right)$, which is repulsive.
\item Finally, for any $(x,y)\in\Gamma\setminus \mathcal{Q}$ there exists some $n$ such that $F^n(x,y)\in \mathcal{P}$. 
\end{itemize}
}

In this section, we provide a different and simple proof of the fact that the set of preimages of the plateau $\overline{R_2S}$ has full measure. Of course, this fact, also follows from the above result in \cite{CGMM}, since the $7$-periodic orbit $\mathcal{P}$, that attracts all the orbits (except the repelling orbit $\mathcal{Q}$) 
is the $\omega$-limit of the plateau.

\begin{propo}\label{p:propofull}
The set of preimages by $F^{7n}$ of the plateau $\overline{R_2 S}$,  namely $\bigcup_{n=0}^{\infty} F^{-7n}|_{\Gamma}(\overline{R_2S})$, has full Lebesgue measure in $\Gamma$. 
\end{propo}

As a consequence of the above result there is a full measure set of initial conditions in $\Gamma$ such that their orbits tend to the $\omega$-limit set of this plateau, $\mathcal{P}=\omega\left(\overline{R_2 S}\right)$ and the proof of Theorem~\ref{t:teonou} when $a<0$ and $b/a<-2$ follows.

Before starting its proof we recall some of the basic aspects of the dynamics of $F|_{\Gamma}$  in this case. The graph in Figure \ref{ff:1} is a topological circle and $F\vert_{\Gamma}$ is non-decreasing. The interval $\overline{R_2S}$ is the unique plateau of the graph. Indeed, $\overline{R_2S}\twoheadrightarrow P_1$ (where the symbol $\twoheadrightarrow$ means ``collapses to'') and, of course, also $F(R_2)=F(S)=P_1$. This implies that the 
 the $7$-periodic orbit $F(P_i)=P_{i+1}$ for $i=1\ldots 6$, is the $\omega$-limit set $\omega\left(\overline{R_2 S}\right)$.

Denoting  $A:=\overline{P_{1}P_{2}}$, $B:=\overline{P_{2}P_{3}}$, $C:=\overline{P_{3}P_{4}}$, $D:=\overline{P_{4}P_{5}}$, $E:=\overline{P_{5}R_1}$, $G:=\overline{P_{6}R_2}$ and $H:=\overline{P_{1}S},$ we obtain the following oriented graph:

\begin{center}
\begin{tikzcd}\label{e:graphA}
	A \arrow[r] & B \arrow[r] & C \arrow[r] & D \arrow[r] & E \arrow[r] & G \arrow[r] & H \arrow[llllll, bend right]
\end{tikzcd}
\end{center}

We do not include here the plateau and also the interval $\overline{P_6R_1}$ since $\overline{P_6R_1}\rightarrow \overline{P_7R_2}\twoheadrightarrow P_1$.

\begin{proof}[Proof of Proposition \ref{p:propofull}]
According with the above directed graph, $F^7$ leaves invariant each of the intervals  $A,\ldots,$ $H$.
Let us consider the edge $A$ delimited by the points $P_1$ and $P_2$ (which, as mentioned before, belong to $\omega\left(\overline{R_2 S}\right)$).  This edge can be parametrized as
$A=\left\{(-b-2,y),\,-3\leq y\right.$ $\left.\leq -1\right\}. $
A computation shows that $F^7_{|A}=F^7(-b-2,y)=(-b-2,f(y))$, where
$$
f(y)=\left\{\begin{array}{ll}
-3& -3\leq y\leq -5/4,\\
16y+17 & -\frac{5}{4}< y< -\frac{9}{8},\\
-1 & -\frac{9}{8}\leq y\leq -1,
\end{array}\right.
$$
which is depicted in  Figure \ref{grafic0}. The map $f(y)$ has a unique fixed point $p$ in $[-3,-1]$ which is a repellor.

\begin{figure}[ht]
\centering
\footnotesize
\begin{lpic}[l(2mm),r(2mm),t(2mm),b(2mm)]{grafic0-v2(0.42)}
\lbl[l]{-2,7; $-3$}
\lbl[l]{160,7; $-\frac{5}{4}$}
\lbl[l]{172,7; $-\frac{9}{8}$}
\lbl[l]{183,7; $-1$}

\lbl[l]{-8,18; $-1$}
\lbl[l]{-8,189; $-3$}
\end{lpic}
\caption{The graph of the function $f(y)$ such that $F^7_{|A}=F^7(1,y)=(1,f(y))$.}\label{grafic0}
\end{figure}

Set $\mathcal{A}_n=F^{-7n}(\overline{R_2 S})\cap A$.  From the expression of $f(y)$ we notice that $\mathcal{A}_1=\mathcal{A}_{1,1}\cup \mathcal{A}_{1,2}$, where the sets $\mathcal{A}_{1,1}=\{(-b-2,y), y\in[-\frac{9}{8},-1]\}$ and $\mathcal{A}_{1,2}=\{(-b-2,y), y\in[-3,-5/4]\}$ are such that
$\mathcal{A}_{1,1}\twoheadrightarrow P_1$ and $\mathcal{A}_{1,2}\twoheadrightarrow P_2$ by $F^7$. Hence,
\begin{equation}\label{e:equacionsdelesA1}
\mathcal{A}_1=F^{-7}(\overline{R_2 S})\cap A
=F^{-7}(P_1\cup P_2)\cap A=\mathcal{A}_{1,1}\cup \mathcal{A}_{1,2}.
\end{equation}
Again from the expression of $f(y),$ is clear that the orbit of every initial condition $y\in[-3,-1]\setminus\{p\}$  reaches $\mathcal{A}_1$ in a finite number of iterates. So $$\bigcup_{n=1}^\infty \mathcal{A}_n=\bigcup_{n=1}^\infty F^{-7(n-1)}(\mathcal{A}_1)=A\setminus \{p\}.$$

A similar result can be obtained by applying the same reasoning to the edges $B=F(A), C=F^2(A), \ldots,$ and $H=F^6(A)$.

Remember that $\overline{P_6R_1}\rightarrow \overline{P_7R_2}\twoheadrightarrow P_1$, hence the preimages of $P_1$ also fully cover these intervals and, therefore, $\lim\limits_{n \to \infty} \ell\left(F^{-n}(\overline{SR_2}) \cap \Gamma\right)=\ell(\Gamma)$, and the result follows.
\end{proof}

To be honest, however, we believe that the proof Proposition~\ref{p:propofull} is not easily extendable to cases with more complex graph geometries.  Nonetheless,  we have obtained the same result 
for the graphs and maps studied in next section, see Proposition \ref{p:x}.  The proof for these new cases is a direct consequence of a result in \cite{BMT}. 
However, and while being direct, this proof cannot be considered elementary, as the result itself does not have an elementary proof.

\subsection{Rest of the cases}\label{ss:lebesalphabeta}

In this section we prove  Theorem~\ref{t:teonou} for the cases not covered by Proposition~\ref{p:propofull}. As usual we can restrict our attention to the case $a=-1.$ More concretely, it suffices to prove next proposition:

\begin{propo}\label{p:x}
	When $a=-1,$ the set of preimages of the plateaus of the invariant graphs~$\Gamma$ associated with the maps $F|_{\Gamma}$ for 
	$b\in [-112/137,-13/16]\cup [603/874,563/816]$ 
	has full Lebesgue measure in~$\Gamma$. 
\end{propo}

\begin{proof}
In the previous section we have showed that the maps $g_2(x)$ and $k_1(x),$ which condensate all the dynamic features of the map $F|_{\Gamma},$  are conjugate to the maps $g_4(x)$ and $k_2(x),$ that belong to the family of trapezoidal maps  $T_{X,Y,Z}$ studied in \cite{BMT}. More concretely:
\begin{itemize}

\item   $g_2(x)$ is conjugated to $g_4=T_{1/16,1/8,Z}$ that corresponds to 	$b\in [-112/137,-13/16];$  and

\item    $k_1(x)$ is conjugated to $k_2=T_{1/16,1/16,Z}$ that corresponds to $b\in[603/874,563/816],$
\end{itemize}
where we have omitted the explicit values of $Z=Z(b),$ for the sake of simplicity, and because they have already been given above.
 One of the main results in \cite{BMT}, which the authors refer to  as \emph{Main fact}, is that  \emph{the preimages of the interval of constancy of any trapezoidal map $g(x)=T_{X,Y,Z}$ has full Lebesgue measure throughout the entire interval $[0,1]$}.

The constancy intervals are the preimages under $F^{-7}$ of the plateaus in $\Pi$, in the first case, and the preimages under $F^{-6}$ of the plateaus in $\Sigma,$ in the second one.  Consequently, the preimages of the plateaus of the corresponding graphs $\Gamma$ have full Lebesgue measure on those graphs, as we wanted to prove.
\end{proof}

\subsection*{Acknowledgments} 
This work is supported by
Ministry of Science and Innovation--State Research Agency of the
Spanish Government through grants PID2022-136613NB-I00   and PID2023-146424NB-I00. It is also supported by the grants 2021-SGR-00113 
and 
2021-SGR-01039 from AGAUR of Generalitat de Catalunya.

\bibliographystyle{plain}

\end{document}